\pdfoutput=1
\documentclass{amsart}

\usepackage{stackrel,amssymb}
\usepackage{slashed}
\usepackage{braket}
\usepackage{mathrsfs}
\usepackage{ifthen}
\usepackage{here}
\usepackage{todonotes}
\usepackage{tikz}
\usetikzlibrary{patterns,decorations.pathreplacing,calligraphy}
\usepackage{comment} 
\usepackage{mleftright}
\usepackage[pagebackref,hypertexnames=false]{hyperref} 
\usepackage{cleveref}
 \usepackage[all]{xy}
 \usepackage{amscd}
 \usepackage[alphabetic,backrefs,msc-links]{amsrefs}
 \usepackage{color}
 \usepackage{enumitem}
\newlist{steps}{enumerate}{1}
\setlist[steps, 1]{label = Step \arabic*:}
\usepackage[abs]{overpic}
\usepackage{tikz-cd}
\usepackage{pinlabel}
\usepackage{amsmath, amsthm,verbatim,amsfonts, graphicx, enumerate}

\usepackage{geometry}
\makeatletter
\DeclareRobustCommand\widecheck[1]{{\mathpalette\@widecheck{#1}}}
\def\@widecheck#1#2{%
   \setbox\z@\hbox{\m@th$#1#2$}%
   \setbox\tw@\hbox{\m@th$#1%
      {%
         \vrule\@width\z@\@height\ht\z@
         \vrule\@height\z@\@width\wd\z@}$}%
   \dp\tw@-\ht\z@
   \@tempdima\ht\z@ \advance\@tempdima2\ht\tw@ \divide\@tempdima\thr@@
   \setbox\tw@\hbox{%
      \raise\@tempdima\hbox{\scalebox{1}[-1]{\lower\@tempdima\box\tw@}}}%
   {\ooalign{\box\tw@ \cr \box\z@}}}
\makeatother

\theoremstyle{plain}
\newtheorem{thm}{Theorem}[section]
\crefname{thm}{Theorem}{Theorems}
\Crefname{thm}{Theorem}{Theorems}
\newtheorem{prop}[thm]{Proposition}
\crefname{prop}{Proposition}{Propositions}
\Crefname{prop}{Proposition}{Propositions}
\newtheorem{lem}[thm]{Lemma}
\crefname{lem}{Lemma}{Lemmas}
\Crefname{lem}{Lemma}{Lemmas}
\newtheorem{cor}[thm]{Corollary}
\crefname{cor}{Corollary}{Corollaries}
\Crefname{cor}{Corollary}{Corollaries}

\crefname{claim}{Claim}{Claims}
\Crefname{claim}{Claim}{Claims}

\crefname{property}{Property}{Properties}
\Crefname{property}{Property}{Properties}

\crefname{problem}{Problem}{Problems}
\Crefname{problem}{Problem}{Problems}

\crefname{conjecture}{Conjecture}{Conjecture}
\Crefname{conjecture}{Conjecture}{Conjecture}

\theoremstyle{definition}
\newtheorem{defn}[thm]{Definition}
\crefname{defn}{Definition}{Definitions}
\Crefname{defn}{Definition}{Definitions}

\crefname{notation}{Notation}{Notations}
\Crefname{notation}{Notation}{Notations}

\crefname{convention}{Convention}{Conventions}
\Crefname{convention}{Convention}{Conventions}

\crefname{cond}{Condition}{Conditions}
\Crefname{cond}{Condition}{Conditions}

\crefname{assum}{Assumption}{Assumptions}
\Crefname{assum}{Assumption}{Assumptions}

\crefname{conj}{Conjecture}{Conjectures}
\Crefname{conj}{Conjecture}{Conjectures}

\crefname{claim1}{Claim}{Claims}
\Crefname{claim1}{Claim}{Claims}
\Crefname{ques}{Question}{Question}

\crefname{que}{Question}{Question}
\Crefname{que}{Question}{Question}

\theoremstyle{remark}
\newtheorem{rem}[thm]{Remark}
\crefname{rem}{Remark}{Remarks}
\Crefname{rem}{Remark}{Remarks}

\crefname{ex}{Example}{Examples}
\Crefname{ex}{Example}{Examples}

\crefname{section}{Section}{Sections}
\Crefname{section}{Section}{Sections}
\crefname{subsection}{Subsection}{Subsections}
\Crefname{subsection}{Subsection}{Subsections}
\crefname{figure}{Figure}{Figures}
\Crefname{figure}{Figure}{Figures}

\newcommand{\Z}{\mathbb{Z}}

\newcommand{\mbar}[1]{{\ooalign{\hfil#1\hfil\crcr\raise.167ex\hbox{--}}}}

\newcommand*{\QEDB}{\null\nobreak\hfill\ensuremath{\square}}%

     \RequirePackage{rotating}                   
    \def\HMt{%
       \setbox0=\hbox{$\widehat{\mathit{HM}}$}
       \setbox1=\hbox{$\mathit{HM}$}
       \dimen0=1.1\ht0
       \advance\dimen0 by 1.17\ht1
       \smash{\mskip2mu\raise\dimen0\rlap{%
          \begin{turn}{180}
              {$\widehat{\phantom{\mathit{HM}}}$}
           \end{turn}} \mskip-2mu    
                \mathit{HM}
                    }{\vphantom{\widehat{\mathit{HM}}}}{}}

\title{Smooth Concordance of Cables of the Figure-Eight Knot}

\author{Sungkyung Kang}
\address{Mathematical Institute, University of Oxford, United Kingdom}
\email{sungkyung38@icloud.com}

\author{JungHwan Park}
\address{Department of Mathematical Sciences, KAIST, Republic of Korea}
\email{jungpark0817@kaist.ac.kr}

\author{Masaki Taniguchi} 
\address{Department of Mathematics, Kyoto University, Japan}
\email{taniguchi.masaki.7m@kyoto-u.ac.jp}

\begin{document}

\begin{abstract}
We prove that every nontrivial cable of the figure-eight knot has infinite order in the smooth knot concordance group. Our main contribution is a uniform proof that applies to all $(2n,1)$-cables of the figure-eight knot. To this end, we introduce a family of  concordance invariants $\kappa_R^{(k)}$, defined via $2^k$-fold branched covers and real Seiberg--Witten Floer $K$-theory. These invariants generalize the real $K$-theoretic Fr\o yshov invariant developed by Konno, Miyazawa, and Taniguchi.
\end{abstract}

\maketitle

\section{Introduction}

A knot in $S^3$ is called \textit{slice} if it bounds a smooth, properly embedded disk in $B^4$.
Moreover, if the radial function restricted to the disk has no local maxima, the knot is called \textit{ribbon}.
By definition, every ribbon knot is necessarily slice. Whether the converse holds remains a well-known open problem, first posed by Fox~\cite[Problem~25]{Fox:1962-1}, and commonly referred to as the \textit{slice--ribbon conjecture}.

One of the most well-known potential counterexamples is constructed using the figure-eight knot as follows.  
Recall that a knot in a homology sphere is called \textit{homotopically ribbon} if it bounds a smooth disk in a homology ball, where the inclusion of the knot complement into the disk complement induces a surjective map on the fundamental group.  
Note that every ribbon knot is homotopically ribbon.
The seminal work of Casson and Gordon~\cite[Theorem~5.1]{CG83} characterizes fibered knots in homology spheres that are homotopically ribbon by examining their closed monodromy.  
This characterization was later used by Miyazaki~\cite[Theorem~8.6]{Miyazaki:1994-1} to show that many fibered knots are not ribbon, potentially providing counterexamples to the slice--ribbon conjecture.  
Among these, one of the most well-known families arises from finite connected sums of nontrivial cables of the figure-eight knot.

The cables of the figure-eight knot have attracted considerable attention, not only because it is one of the simplest examples studied by Miyazaki, but also because of its subtle behavior in four-dimensional topology. 
In particular, Cochran, following earlier work of Fintushel and Stern~\cite{Fintushel-Stern:1984-1}, observed that the figure-eight knot and all of its $(p,1)$-cables bound a smoothly and properly embedded disk in a rational homology ball.  
Furthermore, it has been shown that all $(2n,1)$-cables are algebraically slice~\cite{kawauchi1, Cha:2007-1, Kim-Wu:2018-1}.  
These results make determining the smooth sliceness of such knots especially challenging, even in the simplest case of the $(2,1)$-cable of the figure-eight knot (see~\cite[Section~1]{DKMPS:2024} and~\cite[Section~1]{KPT:2024} for surveys).  
In this article, we completely resolve the smooth sliceness problem for the entire family of these cables and their connected sums.

\begin{thm}\label{thm:main}
    Every nontrivial cable of the figure-eight knot has infinite order in the smooth knot concordance group.
\end{thm}

We remark that this provides the first known example of a torsion knot whose every nontrivial cable is not smoothly slice, thereby offering the first torsion example that supports a question of Miyazaki, posed in~\cite[Question~3]{Miyazaki:1994-1} (see also~\cite[Conjecture~1.4]{Meier:2017} and~\cite[Conjecture~1]{Kang-Park:2022}), which asks whether a nontrivial cable of a non-slice knot can ever be smoothly slice.
In fact, our result \Cref{thm: kappa inequality}, combined with \cite[Theorem~1.1]{Kang-Park:2022}, which uses the involutive knot Floer package, applies more broadly to any Floer thin knot, including all alternating and quasi-alternating knots, that can be transformed into a slice knot by performing full negative twists along two disjoint disks: one intersecting the knot algebraically once, and the other intersecting it algebraically three times.
There are infinitely many such examples representing elements of order two in the smooth concordance group, including the figure-eight knot~\cite[Remark~2.6]{ACMPS:2023-1}.

\subsection{Previous Work}

We begin by briefly recalling the known results related to this problem.  
Throughout this article, we denote by $K_{p,q}$ the $(p,q)$-cable of a knot $K$, where $p$ indicates the longitudinal winding.  
Our focus will be on $(4_1)_{p,1}$, where $4_1$ denotes the figure-eight knot,  
since it is straightforward to verify (for instance, by using the Levine--Tristram signature~\cite{Levine:1969, Tristram:1969}) that all other cables of $4_1$ have infinite order in the smooth knot concordance group, which we denote by $\mathcal{C}$.

The first breakthrough was made for odd values of $p$. Hom, Kang, Park, and Stoffregen~\cite{HKPS:2022-1} proved that $(4_1)_{2n+1,1}$ has infinite order in $\mathcal{C}$ for each nonzero integer $n$, using the involutive knot Floer package introduced by Hendricks and Manolescu~\cite{Hendricks-Manolescu:2017-1}, which is a refinement of the knot Floer complex defined by Ozsváth and Szabó~\cite{ozsvath2004holomorphicknot} and Rasmussen~\cite{Rasmussen:2003-1}.  
Their result also crucially relies on the bordered Floer homology of Lipshitz, Ozsv\'{a}th, and Thurston~\cite{LOT}, as well as its interpretation in terms of immersed curves developed by Hanselman, Rasmussen, and Watson~\cite{HRW:2022, HRW:2023, HRW:2024}.  Moreover, they showed that these knots are linearly independent in $\mathcal{C}$.

For even values of $p$, the first result was for $(4_1)_{2,1}$, which was shown to have infinite order in $\mathcal{C}$ by Dai, Kang, Mallick, Park, and Stoffregen~\cite{DKMPS:2024}.  
They proved that the 2-fold branched cover of $(4_1)_{2,1}$ does not bound an equivariant $\mathbb{Z}_2$-homology ball,  
that is, a $\mathbb{Z}_2$-homology ball over which the deck transformation extends as a smooth involution.  
Their proof relies on the involutive Heegaard Floer package~\cite{Hendricks-Manolescu:2017-1},  
an enhancement of the usual Heegaard Floer homology~\cite{Ozsvath-Szabo:2004-2, Ozsvath-Szabo:2004-1},  
and builds on ideas and results from~\cite{alfieri2020connected, Dai-Hedden-Mallick:2023, dms, mallick-surgery}.

We remark that this strategy provides a strictly stronger refinement of the classical obstruction to sliceness,  
which traditionally focuses on obstructing the existence of a $\mathbb{Z}_2$-homology ball  
(see, for example,~\cite{CG88, En95, Lisca:2007-1, Lisca:2007-2, MO07, Greene-Jabuka:2011-1, Hedden-Livingston-Ruberman:2012-1, HKL16, Hedden-Pinzon:2021}).  
In this case, although the 2-fold branched cover of $(4_1)_{2,1}$ does bound a $\mathbb{Z}_2$-homology ball,  
the authors show that the equivariant obstruction prevents it from bounding one in a way compatible with the deck transformation.  
Alternative proofs of the non-sliceness of $(4_1)_{2,1}$ were subsequently established in~\cite[Theorem~2.1]{ACMPS:2023-1} and~\cite[Corollary~1.20]{KMT23f}.

More recently, a related strategy was applied in~\cite{KPT:2024}, where the authors used a different technique based on real Seiberg--Witten theory,  
combined with a topological construction from~\cite[Section~2]{ACMPS:2023-1}, to conclude that $(4_1)_{4n+2,1}$ is not slice for each integer $n$.  
There are several variants of real Seiberg--Witten theory; for examples, see~\cite{TW09, Na13, Nak15, KMT21, Ka22, Ji22, KMT:2023, Mi23, Li23, BH24, baraglia2025exotic}.
In~\cite{KPT:2024}, the authors use the concordance invariants 
\[
\delta_R(K), \quad \underline{\delta}_R(K), \quad \overline{\delta}_R(K) \in \frac{1}{16} \mathbb{Z}.
\]
of Konno, Miyazawa, and Taniguchi~\cite{KMT:2023}, together with the cohomological real Fr\o yshov inequalities, to show that the 2-fold branched covers of each $(4_1)_{4n+2,1}$ fail to bound equivariant $\mathbb{Z}_2$-homology balls, thereby establishing their non-sliceness. 
The concordance invariants are derived by considering fixed-point sets of certain Fr\o yshov-type invariants on the fixed-point spectrum of an order 2 subgroup $\langle I \rangle \subset O(2)$ acting on Manolescu’s Seiberg--Witten Floer homotopy type~\cite{Ma03} of the 2-fold branched cover of a knot.

\subsection{Our contribution}
To establish the main theorem, it remains to prove that each knot of the form $(4_1)_{2n,1}$ with $n \geq 2$ has infinite order in $\mathcal{C}$.  
We briefly remark on previous techniques applied to this family of knots.  

Although the involutive Heegaard Floer theory used in~\cite{HKPS:2022-1, DKMPS:2024} is a powerful tool,  
the obstructions arising from the involutive knot Floer package vanish in this setting~\cite[Remark~1.4]{HKPS:2022-1}.  
Furthermore, with current techniques, the computations required to extract obstructions as in~\cite{DKMPS:2024} from involutive Heegaard Floer theory remain out of reach for this family.

In~\cite{ACMPS:2023-1}, the obstruction to the smooth sliceness of $(4_1)_{2,1}$ relies on a result of Bryan~\cite[Corollary~1.7]{Bryan:1998-1},  which establishes that $g_{2\mathbb{CP}^2}(2,6) = 10$, meaning that the minimal genus of a smoothly embedded closed oriented surface in $2\mathbb{CP}^2$ representing the homology class $(2,6)$ is $10$.
To extend this strategy to larger values of $n$, one would need the identity
\[
g_{2\mathbb{CP}^2}(2n,6n) = g_{\mathbb{CP}^2}(2n) + g_{\mathbb{CP}^2}(6n),
\]
as noted in~\cite[Corollary~2.9]{ACMPS:2023-1}.  
However, this identity is not known to hold in general; indeed, it is plausible that it may fail for certain values of $n$.  
Currently, the only case where it has been verified is when $n = 1$. A similar limitation applies to approaches based on families of Seiberg--Witten theory combined with the construction from~\cite{ACMPS:2023-1}, including those used in~\cite{KMT23f}.

Lastly, in the strategy of~\cite{KPT:2024}, which applies the invariants $\delta_R$, $\underline{\delta}_R$, and $\overline{\delta}_R$ from real Seiberg--Witten theory,  
a key step is to consider the 2-fold branched cover of a surface cobounded by $(4_1)_{2n,1}$ and $T_{2n,-10n+1}$ in a twice-punctured $2\mathbb{CP}^2$,  
constructed in~\cite{ACMPS:2023-1} (this surface will also play a central role in our article).  
It turns out that this 2-fold branched cover is spin if and only if $n$ is odd, and this spiness is essential to their argument.  
Moreover, the assumptions on $b^+$ in the cohomological real Fr\o yshov inequalities~\cite{KMT:2023} for the invariants $\delta_R$, $\underline{\delta}_R$, and $\overline{\delta}_R$  
make it difficult to extend their method to show that these knots have infinite order in $\mathcal{C}$, even in cases where $n$ is odd.

In order to overcome these difficulties, we use higher-order branched covering spaces together with $\mathbb{Z}_4$-equivariant Floer $K$-theory.  
More precisely, for knots of the form $(4_1)_{2n,1}= (4_1)_{2^k \cdot m, 1}$ with $k \geq 1$ and $m$ odd,  
we consider their $2^k$-fold branched covering spaces.  
In particular, we introduce a family of smooth concordance invariants
\[
\kappa_R^{(k)}\colon \mathcal{C} \to \frac{1}{16} \mathbb{Z},
\]
parametrized by integers $k \geq 1$, and normalized so that the unknot maps to zero.  
These invariants are constructed using $\mathbb{Z}_4$-equivariant $K$-theory applied to the real version of Floer homotopy type of the $2^k$-fold branched cover, equipped with the natural involution induced by composing the deck transformation $2^{k-1}$ times.  
These concordance invariants are higher-order analogues of the real Fr\o yshov invariant $\kappa_R$ introduced by Konno, Miyazawa, and Taniguchi~\cite{KMT21}.  
Indeed, the invariant $\kappa_R^{(k)}$ coincides with $\kappa_R$ applied to the knot that arises as the lift of the original knot to the $2^{k-1}$-fold branched cover,  
so that $\kappa_R^{(1)} = \kappa_R$.

In \Cref{thm: kappa inequality}, we prove that
\[
\frac{1}{2} \leq \kappa_R^{(k)} \left( c  (4_1)_{2^k \cdot m, -1} \right)
\]
for all positive integers $k$, $m$, and $c$, where $m$ is odd and $cK$ denotes the $c$-fold connected sum of a knot $K$.
This inequality follows from two key steps:
\begin{itemize}
    \item Applying real $10/8$ inequality for 4-manifolds with involutions obtained as the $2^k$-fold branched covers along the previously mentioned surface  from \cite{ACMPS:2023-1}. It reduces the computation of $\kappa_R^{(k)}$ for certain torus knots. 
    \item For the computations of $\kappa_R^{(k)}$ for 
 torus knots, we employ Mrowka--Ozsv\'{a}th--Yu's description \cite{MOY96} of the Seiberg--Witten moduli spaces of  Seifert rational homology spheres, combined with the $O(2)$-equivariant lattice Floer homotopy type developed in \cite{KPT:2024} based on \cite{DSS2023}. 
\end{itemize}

We note that the computation of ${\kappa}_R^{(k)}$ for each knot $c (4_1)_{2^k \cdot m, -1}$  
implies that its $2^k$-fold branched cover does not bound any smooth $\mathbb{Z}_2$-equivariant spin $\mathbb{Z}_2$-homology 4-ball,  
with respect to the natural involution induced by composing the deck transformation $2^{k-1}$ times (see \Cref{thm: higher real 10/8} for the precise statement).

\subsection*{Notation and conventions}
While the Neumann-Siebenmann invariant $\bar\mu(Y,\mathfrak{s})$ is defined for plumbed rational homology spheres $Y$ together with a choice of a $\mathrm{spin}^c$ structure $\mathfrak{s}$ on it, when $Y$ is a $\mathbb{Z}_2$-homology sphere and thus carries a unique spin structure $\mathfrak{s}_0$, we will drop it from $\bar\mu(Y,\mathfrak{s}_0)$ and simply write $\bar\mu(Y)$.

\subsection*{Acknowledgements} 

We would like to thank an anonymous referee of our past paper \cite{KPT:2024} for helpful comments, and Jin Miyazawa for helpful discussions.

The first author and the third author gratefully acknowledge support from the Simons Center for Geometry and Physics, Stony Brook University at which many parts of the paper were written.
The second author is partially supported by the Samsung Science and Technology Foundation (SSTF-BA2102-02) and the NRF grant RS-2025-00542968.
The third author was partially supported
by JSPS KAKENHI Grant Number 22K13921 and RIKEN iTHEMS Program.

\section{Some topological computations}

\subsection{An AR negative-definite plumbing graph for $\Sigma(2^k,2^k \cdot m,2^k \cdot 10m - 1)$}

To draw an almost-rational negative-definite plumbing graph for the rational Brieskorn sphere $\Sigma(2^k, 2^k \cdot m, 2^k \cdot 10m - 1)$, where $m$ is a positive odd integer, we first determine its Seifert invariants.  
Using the Seifert algorithm (see, e.g., \cite{Neumann-Raymond:1978-1}), we find that it has one singular fiber of type $(m, -1)$ and $2^k$ singular fibers of type $(2^k \cdot 10m - 1,\ 11 - 2^k \cdot 10m)$.  
The Euler number, which determines the weight of the central node in the resulting plumbing graph, is $-2^k$.

Then we represent the fractions $\frac{m}{1}$ and $\frac{2^k \cdot 10m -1}{2^k \cdot 10m - 11}$ as continued fractions of the following form:
\[
[a_1,\dots,a_n] = a_1 - \cfrac{1}{a_2 - \cfrac{1}{\ddots - \cfrac{1}{a_n}}},
\]
where $a_1,\ldots,a_n$ are positive  integers. It is clear that $\frac{m}{1} = m = [m]$. Furthermore, we have
\[
\frac{2^k \cdot 10m -1}{2^k \cdot 10m - 11} =  [\underbrace{2,\dots,2}_{2^k \cdot m -2},3,\underbrace{2,\dots,2}_{8}].
\]
Therefore an AR negative-definite plumbing graph for $\Sigma(2^k,2^k \cdot m,2^k \cdot 10m-1)$ is given as follows.

\vspace{.2cm}
\[
\begin{tikzpicture}[xscale=1.9, yscale=.8, baseline={(0,-0.1)}]
    \node at (-0.1, 0.3) {$-2^k$};
    \node at (-1, 0.3) {$-m$};
    
    \node at (1, 1.8) {$-2$};
    \node at (2.5, 1.8) {$-2$};
    \node at (3.5, 1.8) {$-3$};
    \node at (4.5, 1.8) {$-2$};
    \node at (6, 1.8) {$-2$};

    \node at (1, -1.2) {$-2$};
    \node at (2.5, -1.2) {$-2$};
    \node at (3.5, -1.2) {$-3$};
    \node at (4.5, -1.2) {$-2$};
    \node at (6, -1.2) {$-2$};
    
    \node at (1, 3.3) {$-2$};
    \node at (2.5, 3.3) {$-2$};
    \node at (3.5, 3.3) {$-3$};
    \node at (4.5, 3.3) {$-2$};
    \node at (6, 3.3) {$-2$};

    \node at (1, -2.7) {$-2$};
    \node at (2.5, -2.7) {$-2$};
    \node at (3.5, -2.7) {$-3$};
    \node at (4.5, -2.7) {$-2$};
    \node at (6, -2.7) {$-2$};
    
    \node at (0, 0) (A0) {$\bullet$};
    \node at (1, 1.5) (A1) {$\bullet$};
    \node at (1, 1.3) (A1b) {};
    \node at (2.5, 1.5) (A2) {$\bullet$};
    \node at (2.5, 1.3) (A2b) {};
    \node at (3.5, 1.5) (A3) {$\bullet$};
    \node at (4.5, 1.5) (A4) {$\bullet$};
    \node at (4.5, 1.3) (A4b) {};
    \node at (6, 1.5) (A5) {$\bullet$};
    \node at (6, 1.3) (A5b) {};

    \node at (-1, 0) (B1) {$\bullet$};
    
    \node at (1, -1.5) (C1) {$\bullet$};
    \node at (1, -1.7) (C1b) {};
    \node at (2.5, -1.5) (C2) {$\bullet$};
    \node at (2.5, -1.7) (C2b) {};
    \node at (3.5, -1.5) (C3) {$\bullet$};
    \node at (4.5, -1.5) (C4) {$\bullet$};
    \node at (4.5, -1.7) (C4b) {};
    \node at (6, -1.5) (C5) {$\bullet$};
    \node at (6, -1.7) (C5b) {};
    
    \node at (1, 3) (D1) {$\bullet$};
    \node at (1, 2.8) (D1b) {};
    \node at (2.5, 3) (D2) {$\bullet$};
    \node at (2.5, 2.8) (D2b) {};
    \node at (3.5, 3) (D3) {$\bullet$};
    \node at (4.5, 3) (D4) {$\bullet$};
    \node at (4.5, 2.8) (D4b) {};
    \node at (6, 3) (D5) {$\bullet$};
    \node at (6, 2.8) (D5b) {};
    \node at (6.3, 3) (D5c) {};
    
    \node at (1, -3) (E1) {$\bullet$};
    \node at (1, -3.2) (E1b) {};
    \node at (2.5, -3) (E2) {$\bullet$};
    \node at (2.5, -3.2) (E2b) {};
    \node at (3.5, -3) (E3) {$\bullet$};
    \node at (4.5, -3) (E4) {$\bullet$};
    \node at (4.5, -3.2) (E4b) {};
    \node at (6, -3) (E5) {$\bullet$};
    \node at (6, -3.2) (E5b) {};
    \node at (6.3, -3) (E5c) {};
    
    \draw (A0) -- (A1);
    \draw[dotted] (A1) -- (A2);
    \draw (A2) -- (A3);
    \draw (A3) -- (A4);
    \draw[dotted] (A4) -- (A5);
    \draw (A0) -- (B1);
    \draw (A0) -- (C1);
    \draw[dotted] (C1) -- (C2);
    \draw (C2) -- (C3);
    \draw (C3) -- (C4);
    \draw[dotted] (C4) -- (C5);
    \draw (A0) -- (D1);
    \draw[dotted] (D1) -- (D2);
    \draw (D2) -- (D3);
    \draw (D3) -- (D4);
    \draw[dotted] (D4) -- (D5);
    \draw (A0) -- (E1);
    \draw[dotted] (E1) -- (E2);
    \draw (E2) -- (E3);
    \draw (E3) -- (E4);
    \draw[dotted] (E4) -- (E5);

    \path (A3) -- (C3) node [font=\huge, midway, sloped] {$\dots$};

    \draw [line width=0.4mm, decorate, decoration = {calligraphic brace,mirror}] (A1b) --  (A2b) node[midway,yshift=-0.8em]{$2^k \cdot m-2$};
    \draw [line width=0.4mm, decorate, decoration = {calligraphic brace,mirror}] (C1b) --  (C2b) node[midway,yshift=-0.8em]{$2^k \cdot m-2$};
    \draw [line width=0.4mm, decorate, decoration = {calligraphic brace,mirror}] (A4b) --  (A5b) node[midway,yshift=-0.8em]{$8$};
    \draw [line width=0.4mm, decorate, decoration = {calligraphic brace,mirror}] (C4b) --  (C5b) node[midway,yshift=-0.8em]{$8$};
    \draw [line width=0.4mm, decorate, decoration = {calligraphic brace,mirror}] (D1b) --  (D2b) node[midway,yshift=-0.8em]{$2^k \cdot m-2$};
    \draw [line width=0.4mm, decorate, decoration = {calligraphic brace,mirror}] (E1b) --  (E2b) node[midway,yshift=-0.8em]{$2^k \cdot m-2$};
    \draw [line width=0.4mm, decorate, decoration = {calligraphic brace,mirror}] (D4b) --  (D5b) node[midway,yshift=-0.8em]{$8$};
    \draw [line width=0.4mm, decorate, decoration = {calligraphic brace,mirror}] (E4b) --  (E5b) node[midway,yshift=-0.8em]{$8$};
    \draw [line width=0.4mm, decorate, decoration = {calligraphic brace}] (D5c) -- (E5c) node[midway,xshift=0.8em]{$2^k$};
    
\end{tikzpicture}
\]
\vspace{.2cm}

From this diagram, we can now compute the Neumann--Siebenmann $\bar{\mu}$-invariant of $\Sigma(2^k, 2^k \cdot m, 2^k \cdot 10m - 1)$, which will be used in later sections. The computation is based on the formula
\[
\bar{\mu} = \frac{\sigma(\Gamma) - w^2}{8},
\]
where $\Gamma$ denotes the plumbing graph and $w$ is the spherical Wu class on $\Gamma$. It is straightforward to verify that $w$ is supported on the central node and on half of the $(-2)$-weighted nodes that appear before each $(-3)$ node. This implies
\[
w^2 = -2^k - 2^k(2^k \cdot m - 2) = 2^k - 2^{2k} \cdot m.
\]
Since the plumbing graph $\Gamma$ is negative definite, its signature is given by the negative of the number of nodes:
\[
\sigma(\Gamma) = -\left(2 + 2^k(2^k \cdot m + 7)\right).
\]
Putting these together, we obtain the following lemma.

\begin{lem} \label{lem:mu-bar-special-case}
The Neumann--Siebenmann invariant of the Brieskorn sphere $\Sigma(2^k, 2^k \cdot m, 2^k \cdot 10m - 1)$ is given by
\[
\bar{\mu}(\Sigma(2^k, 2^k \cdot m, 2^k \cdot 10m - 1)) = -2^k - \frac{1}{4}. \eqno\QEDB
\] 
\end{lem}

\subsection{On Brieskorn spheres of the form $\Sigma(2^k,2^k\cdot m,q)$} \label{subsec: brieskorn k_m_q}


Throughout this subsection, we fix positive integers $k$, $m$, and $q$ such that $m$ and $q$ are odd and coprime. We begin with the following observation:
\[
\Sigma_{2^k}(T_{2^k \cdot m,\, q}) \cong \Sigma(2^k,\, 2^k \cdot m,\, q).
\]
Applying the Seifert algorithm, we find that $\Sigma(2^k,\, 2^k \cdot m,\, q)$ has $2^k + 1$ singular fibers in the case $m>1$:
\[
(m,\,-\alpha), \underbrace{(q,\,-\beta),\,\dots,\,(q,\,-\beta)}_{2^k}.
\]
Here, $\alpha$ and $\beta$ are the unique integers satisfying the following conditions:
\begin{itemize}
    \item $0 < \alpha < m$ and $0 < \beta < q$;
    \item $q\alpha \equiv -1 \pmod{m}$ and $2^k\cdot m\beta \equiv -1 \pmod{q}$.
\end{itemize}
The quotient space $\Sigma(2^k,\, 2^k \cdot m,\, q) / S^1$, where $S^1$ acts by the Seifert action, has genus zero. Its Euler number is given by $-e$, where $e$ is the unique positive integer between $1$ and $2^k$ satisfying
\[
emq - q\alpha - 2^k \cdot m\beta = 1.
\]
Let $a_1,\dots,a_r$ and $b_1,\dots,b_s$ be unique positive integers satisfying
\[
\frac{m}{\alpha} = [a_1,\dots,a_r] \qquad \text{ and } \qquad \frac{q}{\beta} = [b_1,\dots,b_s],
\]
where the right-hand sides denote continued fraction expansions. Then the corresponding plumbing graph $\Gamma_{k,m,q}$, which is negative-definite and almost rational, can be drawn as follows. 
\vspace{.2cm}

\[
\begin{tikzpicture}[xscale=2.4, yscale=1, baseline={(0,-0.1)}]
    \node at (-0.1, 0.3) {$-e$};
    \node at (-1, 0.3) {$-a_1$};
    \node at (-2.5, 0.3) {$-a_r$};
    
    \node at (1, 1.8) {$-b_1$};
    \node at (2.5, 1.8) {$-b_s$};

    \node at (1, -1.2) {$-b_1$};
    \node at (2.5, -1.2) {$-b_s$};
    
    \node at (1, 3.3) {$-b_1$};
    \node at (2.5, 3.3) {$-b_s$};

    \node at (1, -2.7) {$-b_1$};
    \node at (2.5, -2.7) {$-b_s$};
    
    \node at (0, 0) (A0) {$\bullet$};
    \node at (1, 1.5) (A1) {$\bullet$};
    \node at (2.5, 1.5) (A2) {$\bullet$};
    \node at (1.75, 1.5) (A3) {};

    \node at (-1, 0) (B1) {$\bullet$};
    \node at (-2.5, 0) (B2) {$\bullet$};
    
    \node at (1, -1.5) (C1) {$\bullet$};
    \node at (2.5, -1.5) (C2) {$\bullet$};
    \node at (1.75, -1.5) (C3) {};
    
    \node at (1, 3) (D1) {$\bullet$};
    \node at (2.5, 3) (D2) {$\bullet$};
    \node at (2.8, 3) (D2c) {};
    
    \node at (1, -3) (E1) {$\bullet$};
    \node at (2.5, -3) (E2) {$\bullet$};
    \node at (2.8, -3) (E2c) {};
    
    \draw (A0) -- (A1);
    \draw[dotted] (A1) -- (A2);
    \draw (A0) -- (B1);
    \draw[dotted] (B1) -- (B2);
    \draw (A0) -- (C1);
    \draw[dotted] (C1) -- (C2);
    \draw (A0) -- (D1);
    \draw[dotted] (D1) -- (D2);
    \draw (A0) -- (E1);
    \draw[dotted] (E1) -- (E2);

    \path (A3) -- (C3) node [font=\huge, midway, sloped] {$\dots$};

    \draw [line width=0.4mm, decorate, decoration = {calligraphic brace}] (D2c) -- (E2c) node[midway,xshift=0.8em]{$2^k$};
    
\end{tikzpicture}
\]
\vspace{.2cm}

The case $m=1$ is a bit different, as we only have $2^k$ singular fibers of type $(q,-\beta)$, where $\beta$ is the unique integer satisfying $0 < \beta < q$ and $2^k\cdot m \beta = -1 \pmod q$. In this case, the Euler number is given by $-e$, where $e = \frac{2^k\cdot m\beta +1}{q}$. The plumbing graph $\Gamma_{k,m,q}$ in this case is given as follows; note that it is again negative definite and almost rational.\footnote{From now on, we will treat this as a special case of the $m>1$ case, with $r=0$.} Here, $b_1,\dots,b_s$ are unique positive integers satisfying $\frac{q}{\beta} = [b_1,\dots,b_s]$.
\vspace{.2cm}

\[
\begin{tikzpicture}[xscale=2.4, yscale=1, baseline={(0,-0.1)}]
    \node at (-0.1, 0.3) {$-e$};
    
    \node at (1, 1.8) {$-b_1$};
    \node at (2.5, 1.8) {$-b_s$};

    \node at (1, -1.2) {$-b_1$};
    \node at (2.5, -1.2) {$-b_s$};
    
    \node at (1, 3.3) {$-b_1$};
    \node at (2.5, 3.3) {$-b_s$};

    \node at (1, -2.7) {$-b_1$};
    \node at (2.5, -2.7) {$-b_s$};
    
    \node at (0, 0) (A0) {$\bullet$};
    \node at (1, 1.5) (A1) {$\bullet$};
    \node at (2.5, 1.5) (A2) {$\bullet$};
    \node at (1.75, 1.5) (A3) {};

    \node at (1, -1.5) (C1) {$\bullet$};
    \node at (2.5, -1.5) (C2) {$\bullet$};
    \node at (1.75, -1.5) (C3) {};
    
    \node at (1, 3) (D1) {$\bullet$};
    \node at (2.5, 3) (D2) {$\bullet$};
    \node at (2.8, 3) (D2c) {};
    
    \node at (1, -3) (E1) {$\bullet$};
    \node at (2.5, -3) (E2) {$\bullet$};
    \node at (2.8, -3) (E2c) {};
    
    \draw (A0) -- (A1);
    \draw[dotted] (A1) -- (A2);
    \draw (A0) -- (C1);
    \draw[dotted] (C1) -- (C2);
    \draw (A0) -- (D1);
    \draw[dotted] (D1) -- (D2);
    \draw (A0) -- (E1);
    \draw[dotted] (E1) -- (E2);

    \path (A3) -- (C3) node [font=\huge, midway, sloped] {$\dots$};

    \draw [line width=0.4mm, decorate, decoration = {calligraphic brace}] (D2c) -- (E2c) node[midway,xshift=0.8em]{$2^k$};
    
\end{tikzpicture}
\]
\vspace{.2cm}

For each $i = 1, \dots, r$, we consider a 2-sphere $S^2$, divided into the northern hemisphere $S_{a_i}^+$ and the southern hemisphere $S_{a_i}^-$,  
as well as a disk bundle $p_{a_i} \colon E_{a_i} \rightarrow S^2$ of Euler number $-a_i$.  
Similarly, for each $j = 1, \dots, s$ and an element $\gamma \in \mathbb{Z}_{2^k}$, we consider a 2-sphere $S^2$,  
divided into the northern hemisphere $S_{b_i,\gamma}^+$ and the southern hemisphere $S_{b_i,\gamma}^-$,  
along with a disk bundle $p_{b_i,\gamma} \colon E_{b_i,\gamma} \rightarrow S^2$ of Euler number $-b_i$.  
Finally, we consider one more copy of $S^2$, divided into the northern hemisphere $S_c^+$ and the southern hemisphere $S_c^-$,  
as well as a disk bundle $p_c \colon E_c \rightarrow S^2$ of Euler number $-e$.  
We denote the zero-sections of $p_{a_i}$, $p_{b_i,\gamma}$, and $p_c$ by $S_{a_i}$, $S_{b_i,\gamma}$, and $S_c$, respectively.  
We also abuse notation and write
\[
S_{a_i} = S_{a_i}^+ \cup S_{a_i}^-, \quad S_{b_i,\gamma} = S_{b_i,\gamma}^+ \cup S_{b_i,\gamma}^-, \quad S_c = S_c^+ \cup S_c^-.
\]

We consider an action of $\mathbb{Z}_{2^k}$ on $S_c^-$ by rotation that fixes the south pole,  
and choose a pairwise disjoint collection $\{D_{c,\gamma}^-\}_{\gamma \in \mathbb{Z}_{2^k}}$ of disks in $S_c^-$ such that $\gamma' \cdot D_{c,\gamma} = D_{c,\gamma\gamma'}$ for all $\gamma, \gamma' \in \mathbb{Z}_{2^k}$.  
We also choose disks $D_c^+$, $D_{a_i}^\pm$, and $D_{b_i,\gamma}^\pm$, centered around the north or south pole of $S_c$, $S_{a_i}$, and $S_{b_i,\gamma}$, respectively.  
Since the disk bundles $p_{a_i}|_{D_{a_i}^\pm}$, $p_{b_i,\gamma}|_{D_{b_i,\gamma}^\pm}$, $p_c|_{D_c^+}$, and $p_c|_{D_{c,\gamma}^-}$ are trivial,  
we choose trivializations and identify their total spaces with $D^2 \times D^2$,  
where the first factor denotes the base and the second denotes the fiber.  
We refer to these $D^2 \times D^2$ neighborhoods as \emph{elementary disk covers}.

Consider the \emph{swapping map}
\[
sw \colon D^2 \times D^2 \xrightarrow{(x, y) \mapsto (y, x)} D^2 \times D^2.
\]
Using this map, we construct the \emph{plumbed 4-manifold} $W_{\Gamma_{k,m,q}}$,  
which is the smooth 4-manifold defined as follows:
\begin{itemize}
    \item We identify $p_{a_1}^{-1}(D_{a_1}^-)$ with $p_c^{-1}(D_c^+)$ via the swap map.
    \item For each $\gamma \in \mathbb{Z}_{2^k}$, we identify $p_{b_1,\gamma}^{-1}(D_{b_1,\gamma}^+)$ with $p_c^{-1}(D_{c,\gamma}^-)$ via the swap map.
    \item For each $i = 1, \dots, r-1$, we identify $p_{a_i}^{-1}(D_{a_i}^+)$ with $p_{a_{i+1}}^{-1}(D_{a_{i+1}}^-)$ via the swap map.
    \item For each $j = 1, \dots, s-1$ and $\gamma \in \mathbb{Z}_{2^k}$, we identify $p_{b_j,\gamma}^{-1}(D_{b_j,\gamma}^-)$ with $p_{b_{j+1},\gamma}^{-1}(D_{b_{j+1},\gamma}^+)$ via the swap map.
\end{itemize}
The plumbed 4-manifold $W_{\Gamma_{k,m,q}}$ admits a natural $\mathbb{Z}_{2^k}$ (generated by $\tau$) action, defined as follows.
\begin{itemize}
    \item On $p_c^{-1}(D_c^-) = D^2 \times D^2$, $\tau$ acts by rotation on the base $D^2$ and as the identity on the fiber $D^2$.
    \item On any other elementary disk cover $D^2 \times D^2$, $\tau$ acts by  
    \[
    D^2 \times D^2 \xrightarrow{(x, y) \mapsto \left( e^{\frac{2\pi i u}{2^k}} x,\ e^{\frac{2\pi i v}{2^k}} y \right)} D^2 \times D^2
    \]
    for some $u, v \in \{0, \dots, 2^k - 1\}$.
    \item It is straightforward to show that there exists a unique way to assign values of $u$ and $v$ on each elementary disk cover such that they are compatible with the action of $\tau$ on all adjacent elementary disk covers.
\end{itemize}
We can also define a smooth $S^1$-action on $W_{\Gamma_{k,m,q}}$ in a very similar way:  
we start with the action of $S^1$ on $p_c^{-1}(D_c^-)$ by fiber rotation, and extend it to every other elementary disk cover using actions of the form
\[
e^{i\theta} \cdot (x, y) = \left( e^{ik_1\theta} x,\ e^{ik_2 \theta} y \right), \quad k_1, k_2 \in \mathbb{Z}.
\]
As in the case of the $\mathbb{Z}_{2^k}$-action, it is shown in~\cite[Section 2.2]{orlik2006seifert} that there is a unique way to assign values of $k_1$ and $k_2$ to each elementary disk piece (except $p_c^{-1}(D_c^-)$, where the $S^1$-action is already prescribed) in a compatible way,  
and the induced $S^1$-action on the boundary $\partial W_{\Gamma_{k,m,q}}$ is the Seifert action on $\Sigma(2^k, 2^k \cdot m, q)$.  
Since the actions of $\mathbb{Z}_{2^k}$ and $S^1$ on $W_{\Gamma_{k,m,q}}$ commute on the elementary disk cover $p_c^{-1}(D_c^-)$, it follows that they commute everywhere on $W_{\Gamma_{k,m,q}}$.  
Hence, after taking the quotient by the $S^1$-action, we obtain a smooth $\mathbb{Z}_{2^k}$-action on the quotient 3-orbifold $W_{\Gamma_{k,m,q}}/S^1$.

Consider the induced action of $\mathbb{Z}_{2^k}$ on the boundary 2-orbifold $\partial W_{\Gamma_{k,m,q}}/S^1$, which is the Seifert quotient $\mathcal{O} = \Sigma(2^k, 2^k \cdot m, q)/S^1$.  
To describe this action, we divide $\mathcal{O}$ into an open cover $U, V \subset \mathcal{O}$, defined as follows:
\begin{itemize}
    \item $U$ is an open disk with a cone point of order $m$ at its center; that is, $U \cong \mathbb{R}^2 / \mathbb{Z}_m$, where a prescribed generator of $\mathbb{Z}_m$ acts on $\mathbb{R}^2$ by multiplication by $e^{-\frac{2\pi \alpha i}{m}}$.
    \item $V$ is an open disk with $2^k$ cone points of order $q$.
    \item $U \cap V \cong S^1 \times \mathbb{R}$.
\end{itemize}
The action of $\mathbb{Z}_{2^k}$ on $\mathcal{O}$ is defined as follows:
\begin{itemize}
    \item Consider the $\mathbb{Z}_{2^k} (= \langle \tau \rangle)$-action on $\mathbb{R}^2$, given by $\tau(v) = e^{\frac{2\pi i}{2^k}} v$.  
    Since this action commutes with the $\mathbb{Z}_m$-action (which acts as multiplication by $e^{-\frac{2\pi \alpha i}{m}}$), it descends to a smooth $\mathbb{Z}_{2^k}$-action on the orbifold $U$.
    \item $\mathbb{Z}_{2^k}$ acts on $V$ by rotating the disk, cyclically permuting the $2^k$ cone points of $V$.
    \item Since these actions agree on $U \cap V \cong S^1 \times \mathbb{R}$, they glue together to give a well-defined smooth $\mathbb{Z}_{2^k}$-action on $\mathcal{O} = U \cup V$,  
    which coincides with the one induced by the $\mathbb{Z}_{2^k}$-action on $W_{\Gamma_{k,m,q}}/S^1$.
\end{itemize}

\begin{rem} \label{rem: holomorphicity of action on orbifold}
    The orbifold $\mathcal{O}$ can be seen as a complex 1-orbifold. One can choose its complex structure such that the $\mathbb{Z}_{2^k}$-action described above is holomorphic.
\end{rem}

We then claim that $(\partial W_{\Gamma_{k,m,q}}, \tau)$ is equivariantly diffeomorphic to $(\Sigma_{2^k}(T_{2^k \cdot m,q}), \text{deck transformation})$.  
To show this, it suffices to prove the following lemma.

\begin{lem} \label{lem: quotient is torus knot}
    The quotient space $\partial W_{\Gamma_{k,m,q}}/\mathbb{Z}_{2^k}$ is diffeomorphic to $S^3$.  
    The image of $\mathrm{Fix}(\tau)$ under the projection map
    \[
    \partial W_{\Gamma_{k,m,q}} \rightarrow \partial W_{\Gamma_{k,m,q}}/\mathbb{Z}_{2^k} \cong S^3
    \]
    is the torus knot $T_{2^k \cdot m, q}$.
\end{lem}

\begin{proof}
For the proof of this lemma, we follow closely the argument in~\cite[Lemma 4.4]{KPT:2024}.  
Since we are only considering the $\mathbb{Z}_{2^k}$-action on the boundary of $W_{\Gamma_{k,m,q}}$,  
it is straightforward to see that replacing the following part of $\Gamma_{k,m,q}$,
\[
\begin{tikzpicture}[xscale=1.5, yscale=1, baseline={(0,-0.1)}]
    \node at (1, 0.3) {$-a_1$};
    \node at (2, 0.3) {$-a_2$};
    \node at (3, 0.3) {$-a_3$};
    \node at (4.5, 0.3) {$-a_n$};
    
    \node at (1, 0) (C') {$\bullet$};
    \node at (2, 0) (C) {$\bullet$};
    \node at (3, 0) (A1) {$\bullet$};
    \node at (3, -0.2) (A1b) {};
    \node at (4.5, 0) (A5) {$\bullet$};
    \node at (4.5, -0.2) (A5b) {};
    
    \draw (C') -- (C) -- (A1);
    \draw[dotted] (A1) -- (A5);

\end{tikzpicture}
\]
corresponding to the fraction $\frac{m}{\alpha}$, with a new leg 

\[
\begin{tikzpicture}[xscale=1.5, yscale=1, baseline={(0,-0.1)}]
    \node at (1, 0.3) {$-a'_1$};
    \node at (2, 0.3) {$-a'_1$};
    \node at (3, 0.3) {$-a'_3$};
    \node at (4.5, 0.3) {$-a'_n$};
    
    \node at (1, 0) (C') {$\bullet$};
    \node at (2, 0) (C) {$\bullet$};
    \node at (3, 0) (A1) {$\bullet$};
    \node at (3, -0.2) (A1b) {};
    \node at (4.5, 0) (A5) {$\bullet$};
    \node at (4.5, -0.2) (A5b) {};
    
    \draw (C') -- (C) -- (A1);
    \draw[dotted] (A1) -- (A5);

\end{tikzpicture}
\]
corresponding to the fraction $\frac{m}{\alpha + (2^k - e)m}$, that is,
\[
\frac{m}{\alpha + (2^k - e)m} = a_1' - \cfrac{1}{a_2' - \cfrac{1}{\ddots - \cfrac{1}{a_n'}}},
\]
while changing the Euler number of the disk bundle $p_c$ to $-2^k$,  
yields a 3-manifold that is $\mathbb{Z}_{2^k}$-equivariantly diffeomorphic to the original one.

Hence, $\partial W_{\Gamma_{k,m,q}}$ can be $\mathbb{Z}_{2^k}$-equivariantly described using the symmetric surgery diagram shown in \Cref{fig:kirby1}.  
Taking the quotient by the $\mathbb{Z}_{2^k}$-action yields the surgery diagram in \Cref{fig:kirby3},  
where the branching set (i.e., the image of $\mathrm{Fix}(\tau)$) is drawn in red.

To verify that this diagram describes $S^3$, we ignore the red curve (as it is not part of the surgery data),  
replace one fractional surgery circle with a chain of integer-framed circles,  
and then apply slam-dunk moves inductively until only a single (unknotted) surgery circle remains.  
Using the Seifert relation
\[
emq - q\alpha - 2^k \cdot m\beta = 1,
\]
we deduce that the resulting surgery slope is of the form $\frac{1}{N}$.  
Hence, the quotient $\partial W_{\Gamma_{k,m,q}}/\mathbb{Z}_{2^k}$ is diffeomorphic to $S^3$.

Following the remainder of the proof in~\cite[Lemma 4.4]{KPT:2024}, we conclude that the knot $K$ is the torus knot $T_{2^k \cdot m, q}$.  
We briefly outline their argument. The first observation is that $0$-surgery on $K$, taken with respect to the blackboard framing, yields a connected sum of two lens spaces.  
This implies, via~\cite[Theorem 1.5]{greene2015space}, that the surgery slope must be $2^k \cdot mq$ with respect to the Seifert framing.  
Additionally, the Seifert invariants of the resulting surgery manifold can be computed using the method in~\cite[Proposition 3.1]{moser1971elementary}.  
Consequently, we obtain the identity
\[
S^3_N(K) \cong S^3_N(T_{2^k \cdot m, q}) \qquad \text{for all integers } N > 1.
\]
Finally, since a large enough slope is a characterizing slope for the torus knot $T_{2^k \cdot m, q}$ by~\cite[Theorem 1.3]{ni2014characterizing},  
we conclude that $K = T_{2^k \cdot m, q}$, as claimed.
\end{proof}

\begin{figure}[h]
\centering
\includegraphics[height=.3\linewidth]{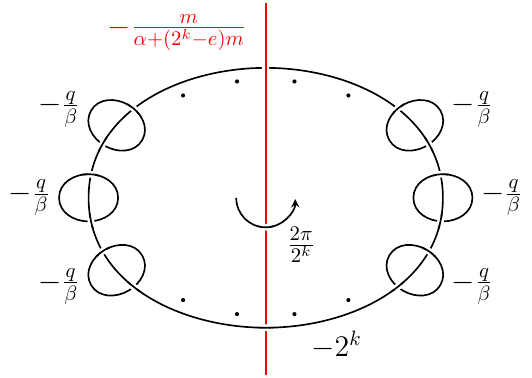}
\caption{A surgery diagram for $\partial W_{\Gamma_{k,m,q}}$. The action of $\tau$, a generator of $\mathbb{Z}_{2^k}$, is given by rotation by $\frac{2\pi}{2^k}$ about the red vertical surgery curve.}
\label{fig:kirby1}
\end{figure}

\begin{figure}[h]
\centering
\includegraphics[height=.28\linewidth]{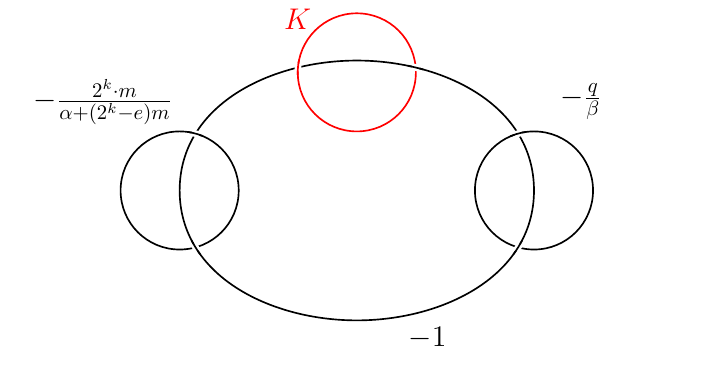}
\caption{A surgery diagram for the quotient 3-manifold $\partial W_{\Gamma_{k,m,q}}/\mathbb{Z}_{2^k}$. The branching set $K$, namely the image of $\mathrm{Fix}(\tau)$ under the projection $\partial W_{\Gamma_{k,m,q}} \to \partial W_{\Gamma_{k,m,q}}/\mathbb{Z}_{2^k}$, is drawn in red.}
\label{fig:kirby3}
\end{figure}


\subsection{Levine-Tristram signatures of torus knots $T_{pm,pmn-1}$}

Given any knot $K$, its Levine--Tristram signature~\cite{Levine:1969, Tristram:1969} is an integer-valued function on the unit circle $S^1$, denoted by
\[
\sigma_K\colon S^1 \to \mathbb{Z}, \qquad \omega \mapsto \sigma_K(\omega).
\]
This function satisfies the following properties:
\begin{itemize}
    \item $\sigma_K(1) = 0$;
    \item $\sigma_K(-1) = \sigma(K)$,
\end{itemize}
where $\sigma(K)$ denotes the classical signature of $K$.
Throughout this paper, we adopt the convention that $\sigma(T_{2,3}) = -2$; that is, positive torus knots have negative signature.

Given a positive integer $n$, we define the following quantity:
\[
\sigma^{(n)}(K) = \sum_{\omega^n = 1} \sigma_K(\omega),
\]
where the sum is taken over all $n$-th roots of unity $\omega \in S^1$. In this subsection, we compute the value of $\sigma^{(p)}(T_{pm, pmn - 1})$ for positive integers $p$, $m$, and $n$, not all equal to $1$.  
This computation will be used in \Cref{subsec: homological computation}.  
We begin by recalling the following result from~\cite[Chapter XII]{Ka87}:

\begin{thm} \label{thm:torus signature}
Let $p, q > 0$ be coprime integers. Define
\[
\Sigma_{p,q} = \left\{ \frac{k}{p} + \frac{l}{q} \,\middle|\, 0 < k < p,\; 0 < l < q \right\} \subset (0,2).
\]
Then for any $x \in (0,1) \setminus \Sigma_{p,q}$, we have
\[
\sigma_{T_{p,q}}(e^{2\pi i x}) = \left| \Sigma_{p,q} \setminus (x,x+1) \right| - \left| \Sigma_{p,q} \cap (x,x+1) \right|. \eqno\QEDB
\] 
\end{thm}

We then observe that $\Sigma_{pm, pmn - 1} \cap \frac{1}{pm} \mathbb{Z} = \emptyset$ and
\[
\left\vert \Sigma_{pm, pmn - 1} \cap \left( \frac{k}{pm}, \frac{k+1}{pm} \right) \right\vert =
\begin{cases}
    0 & \text{if } k = 0; \\
    kn - 1 & \text{if } 0 < k < pm.
\end{cases}
\]
Adding the values of $\left\vert \Sigma_{pm, pmn - 1} \cap \left( \frac{k}{pm}, \frac{k+1}{pm} \right) \right\vert$, we obtain:
\[
\left\vert \Sigma_{pm, pmn - 1} \cap \left( \frac{k}{p}, \frac{k+1}{p} \right) \right\vert =
\begin{cases}
    1 - m + \dfrac{m(m+1)n}{2} & \text{if } k = 0; \\
    -m + knm^2 + \dfrac{m(m+1)n}{2} & \text{if } 0 < k < p.
\end{cases}
\]
Note that we have the following anti-symmetry:
\[
r \in \Sigma_{p,q} \quad \Longleftrightarrow \quad 2 - r \notin \Sigma_{p,q},
\]
and hence we do not need to consider the cases $pm \le k < 2pm$. Using this symmetry and \Cref{thm:torus signature}, we obtain the following identity:
\[
\begin{split}
    \sigma^{(p)}(T_{pm,\,pmn - 1}) 
    &= 2 \cdot \sum_{k=0}^{p-1} (p - 1 - 2k) \cdot \left\vert \Sigma_{pm,\,pmn - 1} \cap \left( \frac{k}{p},\,\frac{k + 1}{p} \right) \right\vert \\
    &= 2(p - 1) + 2 \cdot \sum_{k=0}^{p-1} (p - 1 - 2k)\left( -m + knm^2 + \frac{m(m + 1)n}{2} \right).
\end{split}
\]
For simplicity, we write
\[
N_{m,n} = \sum_{k=0}^{p-1} (p - 1 - 2k)\left( -m + knm^2 + \frac{m(m + 1)n}{2} \right),
\]
so that $\sigma^{(p)}(T_{pm,\,pmn - 1}) = 2(p - 1) + 2N_{m,n}$. 

To compute $N_{m,n}$, we observe that it is a degree-one polynomial in $n$, so the difference $N_{m,n+1} - N_{m,n}$ is independent of $n$. We compute:
\[
\begin{split}
    N_{m,0} &= \sum_{k=0}^{p-1} -m(p-1-2k) = 0, \\
    N_{m,n+1}-N_{m,n} &= \sum_{k=0}^{p-1} (p-1-2k)\left( km^2+\frac{m(m+1)}{2} \right) \\
    &= \frac{m}{2}  \sum_{k=0}^{p-1} (p-1)(m+1) + 2( m(p-1)-m-1 )k - 4mk^2 \\
    &= \frac{m}{2} \left( p(p-1)(m+1) + ((p-2)m-1)p(p-1) - \frac{4mp(p-1)(2p-1)}{3} \right) \\
    &= \frac{mp(p-1)(3(m+1 + (p-2)m - 1) - 2m(2p-1))}{6} \\
    &= -\frac{p(p-1)(p+1)m^2}{6},
\end{split}
\]
which implies that
\[
N_{m,n} = -\frac{p(p - 1)(p + 1)m^2 n}{6}.
\]
Therefore, we obtain the following result.

\begin{lem} \label{lem:signatures}
    For any integers $p, m, n > 0$ that are not all equal to $1$, we have
    \[
    \sigma^{(p)}(T_{pm,\,pmn - 1}) = 2(p - 1) - \frac{p(p - 1)(p + 1)m^2 n}{3}.
    \]
\end{lem}

\begin{proof}
    Using the computation above, we have
    \[
    \sigma^{(p)}(T_{pm,\,pmn - 1}) = 2(p - 1) + 2N_{m,n} = 2(p - 1) - \frac{p(p - 1)(p + 1)m^2 n}{3}. \qedhere
    \]
\end{proof}

\subsection{Branched covers along the disks $D_{k,m}$} \label{subsec: homological computation}

Recall from~\cite{ACMPS:2023-1} that there exists a concordance
\[
S \subset X := 2\mathbb{CP}^2 \smallsetminus \left(\mathring{B}^4 \sqcup \mathring{B}^4\right) \cong (S^3 \times I) \# 2\mathbb{CP}^2 ,
\]
of homology class $(1,3)$ in $H_2(X, \partial X; \mathbb{Z}) \cong H_2(2\mathbb{CP}^2; \mathbb{Z}) = \mathbb{Z} \oplus \mathbb{Z}$, from the $0$-framed figure-eight knot $4_1$ to the $(-10)$-framed unknot (see Figure~\ref{fig:figure8}). Taking the $(2^k \cdot m, 1)$-cable of $S$ then gives a concordance $S_{2^k \cdot m, 1} \subset X$, of homology class $(2^k \cdot m, 2^k \cdot 3m)$, from $(4_1)_{2^k \cdot m, 1}$ to the torus knot
$T_{2^k \cdot m,\, 1 - 2^k \cdot 10m} = -T_{2^k \cdot m,\, 2^k \cdot 10m - 1}$.

\vspace{.2cm}
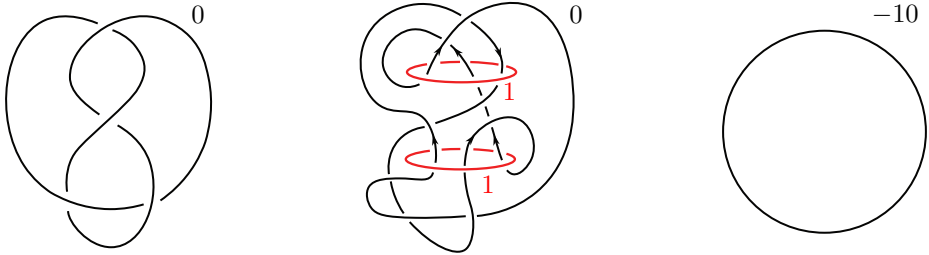
\begin{figure}[h]
\centering
\labellist
\pinlabel $0$ at 88 97
\pinlabel $0$ at 230 97
\pinlabel $-10$ at 350 97
\pinlabel \textcolor{red}{$1$} at 205 68
\pinlabel \textcolor{red}{$1$} at 197 33
\endlabellist
\includegraphics[width=.8\linewidth]{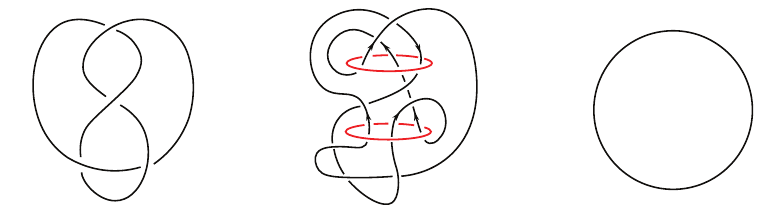}
\caption{A smooth concordance $S$ in the twice-punctured $2\mathbb{CP}^2$ from the figure-eight knot to the unknot.}
\label{fig:figure8}
\end{figure}

Suppose that the knot $(4_1)_{2^k \cdot m,1}$ is smoothly slice. Then it bounds a properly embedded disk in $B^4$. By gluing this disk to the cabled concordance $S_{2^k \cdot m,1}$, we get a smooth disk in $B^4 \# 2\mathbb{CP}^2$, again of homology class $(2^k \cdot m,2^k \cdot 3m)$, that bounds $-T_{2^k \cdot m,2^k \cdot 10m-1}$. We denote this disk by $D_{k,m}$.

In this subsection, we compute the values of $b^+$ and the signature of the $2^k$-fold branched cover $\Sigma_{2^k}(D_{k,m})$ of $B^4 \# 2\mathbb{CP}^2$, branched along $D_{k,m}$. To carry out this computation, we use the following formula, which can be derived from basic algebraic topology and the $G$-signature theorem (see, for example,~\cite[Lemma 2.2]{baraglia2024knot}).

\begin{lem} \label{lem: homological formula}\belowdisplayskip=-11pt
Let $X$ be a closed 4-manifold, and let $S$ be a properly embedded surface in $B^4 \# X$ whose boundary is a knot. Let $n > 0$ be a prime power that divides the homology class $[S] \in H_2(B^4 \# X;\mathbb{Z})$, so that the $n$-fold branched cover $\Sigma_n(S)$ of $B^4 \# X$ branched along $S$ is well-defined. Then we have:
\begin{align*}
    b^+\big(\Sigma_n(S)\big) & = n\, b^+(X) + (n - 1) g(S) - \frac{n^2 - 1}{6n} [S]^2 + \frac{1}{2} \sigma^{(n)}(\partial S), \\
    \sigma\big(\Sigma_n(S)\big) & = n\, \sigma(X) - \frac{n^2 - 1}{3n} [S]^2 + \sigma^{(n)}(\partial S).
\end{align*}\qed
\end{lem}

Using the signature formula, we compute the values of $b^+(\Sigma_{2^k}(D_{k,m}))$, $b^+(\Sigma_{2^{k-1}}(D_{k,m}))$, and $\sigma(\Sigma_{2^k}(D_{k,m}))$.  
Note that the homology class of $D_{k,m}$ is divisible by $2^k$, so the $2^k$-fold branched cover $\Sigma_{2^k}(D_{k,m})$ is well-defined.  
Since the homology class of $D_{k,m}$ is $(2^k \cdot m, 2^k \cdot 3m)$, we compute
\[
[D_{k,m}]^2 = 10 \cdot 2^{2k}  m^2.
\]
Furthermore, applying \Cref{lem:signatures}, we obtain:
\[
\begin{split}
\sigma^{(2^k)} (\partial D_{k,m}) &= -\sigma^{(2^k)}(T_{2^k \cdot m,\, 2^k \cdot 10m - 1}) 
= -2(2^k - 1) + \frac{10 \cdot 2^k (2^k - 1)(2^k + 1)m^2}{3}, \\
\sigma^{(2^{k-1})} (\partial D_{k,m}) &= -\sigma^{(2^{k-1})}(T_{2^k \cdot m,\, 2^k \cdot 10m - 1}) 
= -2(2^{k-1} - 1) + \frac{20 \cdot 2^k (2^{k-1} - 1)(2^{k-1} + 1)m^2}{3}.
\end{split}
\]
Now, applying \Cref{lem: homological formula}, we obtain the following result, which will be used in the subsequent sections.

\begin{lem} \label{lem: homological quantities}
We have
\[
b^+\big(\Sigma_{2^k}(D_{k,m})\big) = 2^k + 1, \qquad b^+\big(\Sigma_{2^{k-1}}(D_{k,m})\big) = 2^{k-1} + 1, \qquad \sigma\big(\Sigma_{2^k}(D_{k,m})\big) = 2.
\]
\end{lem}

\begin{proof}
First, we compute:
\[
\begin{split}
    b^+\big(\Sigma_{2^k}(D_{k,m})\big) &= 2^k \cdot b^+\big(2\mathbb{CP}^2\big) - \frac{2^{2k} - 1}{6 \cdot 2^k} [D_{k,m}]^2 + \frac{1}{2} \sigma^{(2^k)}\big(\partial D_{k,m}\big) \\
    &= 2^{k+1} - \frac{10(2^{2k} - 1)2^{2k} m^2}{6 \cdot 2^k} - (2^k - 1) + \frac{10 \cdot 2^k(2^k - 1)(2^k + 1)m^2}{6} \\
    &= 2^k + 1.
\end{split}
\]
Similarly,
\[
\begin{split}
    b^+\big(\Sigma_{2^{k-1}}(D_{k,m})\big) &= 2^{k-1} \cdot b^+\big(2\mathbb{CP}^2\big) - \frac{2^{2k - 2} - 1}{6 \cdot 2^{k - 1}} [D_{k,m}]^2 + \frac{1}{2} \sigma^{(2^{k - 1})}\big(\partial D_{k,m}\big) \\
    &= 2^k - \frac{10(2^{2k - 2} - 1)2^{2k} m^2}{3 \cdot 2^k} - (2^{k - 1} - 1) + \frac{20 \cdot 2^k(2^{k - 1} - 1)(2^{k - 1} + 1)m^2}{6} \\
    &= 2^{k - 1} + 1.
\end{split}
\]

Finally, the signature is given by:
\[\belowdisplayskip=-11pt
\begin{split}
    \sigma\big(\Sigma_{2^k}(D_{k,m})\big) &= 2^k \cdot \sigma\big(2\mathbb{CP}^2\big) - \frac{2^{2k} - 1}{3 \cdot 2^k} [D_{k,m}]^2 + \sigma^{(2^k)}\big(\partial D_{k,m}\big) \\
    &= 2^{k+1} - \frac{10(2^{2k} - 1)2^{2k} m^2}{3 \cdot 2^k} - 2(2^k - 1) + \frac{10 \cdot 2^k(2^k - 1)(2^k + 1)m^2}{3} \\
    &= 2.
\end{split}
\]
\end{proof}

It is also important that the $2^k$-fold branched cover admits a unique spin structure:
\begin{lem}\label{homology}
     Let $K$ and $K'$ be knots in $S^3$, and let $X$ be an oriented, smooth, compact, connected cobordism from $S^3$ to $S^3$ with $H_1(X; \mathbb{Z}) = 0$. Let $S$ be an oriented, compact, connected, properly and smoothly embedded cobordism in $X$ from $K$ to $K'$, such that the homology class $[S]$ is divisible by $2^k$. Consider the $2^k$-fold branched cover $\Sigma_{2^k}(S)$ of $X$ along $S$, which defines a $\mathbb{Z}_{2^k}$-equivariant cobordism from $\Sigma_{2^k}(K)$ to $\Sigma_{2^k}(K')$. If
     \[
     PD(w_2(X)) = [S]/2^k \pmod 2,
     \]
     where $[S]/2^k$ denotes the unique element of $H^2(X;\mathbb{Z})$ which becomes $[S]$ when multiplied by $2^k$, then $\Sigma_{2^k}(S)$ admits a unique spin structure.
\end{lem}

\begin{proof}
    The existence of a spin structure on $\Sigma_{2^k}(S)$ is established in \cite{nagami2001existence} for the case of closed 4-manifolds, and the argument extends straightforwardly to 4-manifolds with $S^3$ boundaries. Moreover, one can verify that $H^1(\Sigma_{2^k}(S); \mathbb{Z}_2) = 0$ using the transfer sequence for 2-fold coverings. This implies that the spin structure on $\Sigma_{2^k}(S)$ is unique.
\end{proof}

\section{Concordance invariants from higher covers}

\subsection{Notations}
We begin by introducing the following representations of the cyclic group $G = \mathbb{Z}_4$, viewed as the cyclic group of order $4$ generated by $j \in \mathrm{Pin}(2)$:
\[
G = \{1, j, -1, -j\}.
\]
\begin{itemize}
    \item Let $\tilde{\mathbb{R}}$ denote the 1-dimensional real representation of $G$ defined via the surjection $G \to \mathbb{Z}_2 = \{1, -1\}$, where $\mathbb{Z}_2$ acts on $\mathbb{R}$ by scalar multiplication.
    
    \item Similarly, let $\tilde{\mathbb{C}}$ denote the 1-dimensional complex representation of $G$ defined via the same surjection $G \to \mathbb{Z}_2$, with $\mathbb{Z}_2$ acting on $\mathbb{C}$ by scalar multiplication.
    
    \item Let $\mathbb{C}$ denote the 1-dimensional complex representation of $G$ in which $j \in G$ acts as multiplication by $i$.
\end{itemize}
These representations are used to define a  $\mathbb{Z}_4$-equivariant stable homotopy category $\mathfrak{C}_G$ introduced in~\cite{KMT:2023} which contains our $\Z_4$-equivariant real Floer homotopy type. 
We only use certain properties of $\mathfrak{C}_G$ and cohomological or K-theoretic quantities of objects in $\mathfrak{C}_G$, so we do not explain $\mathfrak{C}_G$. See \cite[Section 3.4]{KMT:2023} for more details.

\subsection{The higher real concordance invariants and inequalities}
We introduce the main concordance invariants from real Seiberg--Witten theory. In~\cite{KMT21, KMT:2023}, Konno, Miyazawa, and the third author defined the invariants
\[
{\delta}_R( Y, \mathfrak{s}, \tau), \quad \underline{\delta}_R( Y, \mathfrak{s}, \tau), \quad \overline{\delta}_R( Y, \mathfrak{s}, \tau), \quad \kappa_R( Y, \mathfrak{s}, \tau) \in \frac{1}{16}\mathbb{Z},
\]
where $Y$ is a rational homology 3-sphere, $\tau$ is a smooth involution whose fixed point set is nonempty and of codimension two, and $\mathfrak{s}$ is a spin structure on $Y$ satisfying $\tau^*\mathfrak{s} \cong \mathfrak{s}$. Note that, if $Y$ is an integral homology 3-sphere and $K\subset Y$ is a knot, then for any integer $k>0$, the $2^k$-fold branched  cover $\Sigma_{2^k}(Y,K)$ is a $\mathbb{Z}_2$-homology sphere and thus carries a unique spin structure, which is obviously invariant under deck transformations.

\begin{defn}
Given a knot $K$ in $S^3$, we define:
\begin{align*}
&{\delta}_R^{(k)} (K):=  {\delta}_R( \Sigma_{2^k}(K), \mathfrak{s}_0,  \tau) \in \frac{1}{16} \Z \\
&\underline{\delta}_R^{(k)} (K):=  \underline{\delta}_R( \Sigma_{2^k}(K), \mathfrak{s}_0,  \tau) \in \frac{1}{16} \Z \\
&\overline{\delta}_R^{(k)} (K):=  \overline{\delta}_R( \Sigma_{2^k}(K), \mathfrak{s}_0,  \tau) \in \frac{1}{16} \Z \\
& {\kappa}_R^{(k)} (K):=  {\kappa}_R( \Sigma_{2^k}(K), \mathfrak{s}_0,  \tau) \in \frac{1}{16} \Z 
\end{align*}
and 
\[
SWF_R^{(k)}(K) := [SWF_R ( \Sigma_{2^k}(K), \mathfrak{s}_0,  \tau) ] \in \mathfrak{C}_G
\]
for $k \geq 1$, where $\tau$ is the $2^{k-1}$-fold composition of the deck transformation of the $2^k$-fold branched cover, $\mathfrak{s}_0$ is the unique spin structure on $\Sigma_{2^k}(K)$, and $\mathfrak{C}_G$ is the $\mathbb{Z}_4$-equivariant stable homotopy category introduced in~\cite{KMT:2023}.
\end{defn}

\begin{defn}
    For a positive integer $k$, we say that a knot $K$ is \emph{$k$-strongly spherical} if
    \[
    SWF_R^{(k)}(K) \simeq (\mathbb{C}^r)^+
    \]
    as a $\mathbb{Z}_4$-equivariant stable homotopy type for some $r \in \mathbb{Q}$.
\end{defn}

\begin{rem}
    Clearly, if $K$ is a $k$-strongly spherical knot, then 
    \[
    SWF_R^{(k)}(K) \simeq (\mathbb{C}^{-\kappa_R^{(k)}(K)})^+.
    \]
\end{rem}

The following version of the real $10/8$ inequality holds for the invariants $\kappa_R^{(k)}(K)$:

\begin{thm} \label{thm: higher real 10/8}
    Let $K$ and $K'$ be knots in $S^3$, and let $X$ be an oriented smooth compact connected cobordism from $S^3$ to $S^3$ with $H_1(X; \mathbb{Z}) = 0$. Let $S$ be an oriented, compact, connected, properly and smoothly embedded concordance in $X$ from $K$ to $K'$ such that the homology class $[S]$ is divisible by $2^k$ and 
    \[
    \operatorname{PD}(w_2(X)) = [S]/2^k \pmod{2}.
    \]
    Then we have
    \begin{align*} \label{ineq1}
        -\frac{\sigma(\Sigma_{2^k}(S))}{16} \leq b^+(\Sigma_{2^k}(S)) - b^+(\Sigma_{2^{k-1}}(S)) + \kappa_R^{(k)}(K') - \kappa_R^{(k)}(K),
    \end{align*}
    where $\sigma(W)$ denotes the signature of a 4-manifold $W$.

    Moreover, if 
    \[
    b^+(\Sigma_{2^k}(S)) - b^+(\Sigma_{2^{k-1}}(S)) > 0
    \]
    and $K$ is $k$-strongly spherical, then
    \begin{align*} \label{ineq2}
        -\frac{\sigma(\Sigma_{2^k}(S))}{16} + \frac{1}{2} \leq b^+(\Sigma_{2^k}(S)) - b^+(\Sigma_{2^{k-1}}(S)) + \kappa_R^{(k)}(K') - \kappa_R^{(k)}(K).
    \end{align*}
\end{thm}

\begin{proof}
The first inequality follows from~\cite[Theorem 1.1(iv)]{KMT21}, applied to the $2^k$-fold branched cover along $S$, combined with \Cref{homology}. 
For the second inequality, we apply~\cite[Proposition 6.5]{MPT25} to the involutive spin 4-manifolds given by the $2^k$-fold branched cover along $S$, again using \Cref{homology}.
\end{proof}

Similar inequalities hold for ${\delta}_R^{(k)}$, $\underline{\delta}_R^{(k)}$, and $\overline{\delta}_R^{(k)}$, but we omit them here, as only the inequality for the $\kappa_R^{(k)}$-invariant will be used. These inequalities ensure the following:

\begin{cor}
    For each $k \geq 1$, the invariants ${\delta}_R^{(k)}$, $\underline{\delta}_R^{(k)}$, $\overline{\delta}_R^{(k)}$, and ${\kappa}_R^{(k)}$ descend to well-defined maps
    \[
    {\delta}_R^{(k)},\ \underline{\delta}_R^{(k)},\ \overline{\delta}_R^{(k)},\ {\kappa}_R^{(k)}\colon \mathcal{C} \to \frac{1}{16} \Z,
    \]
    which send the unknot to zero. \qed
\end{cor}

\section{The computation for torus knots}

\subsection{Computation via Mrowka--Ozsv\'ath--Yu}
Throughout this subsection, we fix positive integers $k$, $m$, and $q$ such that $m$ and $q$ are odd and coprime. Let us consider the covering space
\[
Y = \Sigma_{2^k}(T_{2^k \cdot m,\, q}),
\]
with covering transformation $\tau$, and let $\mathfrak{s}_0$ denote the unique spin structure on $Y$. We view $Y$ as the total space of the Seifert fibration
\[
Y = \Sigma(e;\, (a_1, b_1),\, \dots,\, (a_n, b_n)),
\]
where
\[
\{(a_1, b_1), \dots, (a_n, b_n)\} = \{ (m,\,-\alpha),\, \underbrace{(q,\,-\beta),\, \dots,\, (q,\,-\beta)}_{2^k} \}.
\]
Here, $\alpha$ and $\beta$ are the unique integers satisfying the following conditions:
\begin{itemize}
    \item $0 < \alpha < m$ and $0 < \beta < q$;
    \item $q\alpha \equiv 1 \pmod{m}$ and $m\beta \equiv 1 \pmod{q}$.
\end{itemize}

We consider the 2-orbifold $\check{S}$ obtained as the quotient $Y / S^1$. Since the branched covering action commutes with the $S^1$-action, the $\Z_{2^k}$-action descends to a well-defined action on $\check{S}$, denoted by
\[
\bar{\tau} \colon \check{S} \to \check{S}.
\]
Let $\pi$ denote the projection map $Y \to \check{S}$.

\begin{lem}
    In the setting above, the map $\bar{\tau}$ is a rotation of angle $2\pi / 2^k$ about the $z$-axis of $\check{S}$, fixing the north pole corresponding to the singular fiber of type $(m,\,-\alpha)$. When $k = 0$, the action is trivial.
\end{lem}

\begin{proof}
    This follows from \Cref{rem: holomorphicity of action on orbifold} and \Cref{lem: quotient is torus knot}.
\end{proof}

With this setting, we claim the following:

\begin{prop}\label{sphericality}
    We have
    \[
    SWF_R(Y, \mathfrak{s}_0, \tau^{2^{k-1}}) \simeq \left(\mathbb{C}^{m}\right)^+
    \]
    for some number $m \in \mathbb{Q}$, as a stable $\mathbb{Z}_4$-equivariant homotopy type.
\end{prop}

\begin{proof}
We use the description of Seiberg--Witten moduli spaces due to Mrowka--Ozsv\'ath--Yu~\cite{MOY96}. The argument is analogous to the case of odd torus knots given in~\cite[Theorem 3.58]{KMT21}. For precise details, see~\cite{MOY96}.

First, we note that the classification of orbifold line bundles over $\check{S}^2$ is given by an injective map
\[
\operatorname{Pic}^t(\check{S}^2(a_1, \dots, a_n)) \to \mathbb{Q} \oplus \mathbb{Z}_{a_1} \oplus \cdots \oplus \mathbb{Z}_{a_n},
\]
given by
\[
L \mapsto (e, b_1, \dots, b_n),
\]
which is called the \emph{Seifert invariant} of the line bundle $L$. The number $e$ is called the \emph{degree} of $L$, denoted $\deg(L) = e$.
The \emph{orbifold canonical bundle} $K_{\check{S}^2}$ of $\check{S}^2$ is described by the invariant
\[
(0, a_1 - 1, \dots, a_n - 1).
\]

We first fix the following geometric data:
\begin{itemize}
    \item Let $i\eta$ denote the connection $1$-form of the circle bundle $Y \to \check{S}^2$, and let $g_{\check{S}^2}$ be an orbifold metric on $\check{S}^2$ with constant curvature. Then we endow $Y$ with the metric
    \[
    g_Y = \eta^2 + \pi^* g_{\check{S}^2},
    \]
    under which the tangent bundle $TY$ admits an orthogonal splitting
    \begin{align}\label{TY}
        TY = \mathbb{R} \oplus \pi^*(T\check{S}^2).
    \end{align}
    We may take $g_{\check{S}^2}$ to be $\tau$-invariant by averaging.

    \item Let $\nabla$ denote the connection on $Y$ canonically induced from the decomposition~\eqref{TY}.

    \item Let $K_{\check{S}^2}$ denote the canonical line bundle of $\check{S}^2$.

    \item Let $W_c$ denote the spin$^c$ structure on $Y$ defined as
    \[
    \underline{\mathbb{C}} \oplus \pi^* K_{\check{S}^2}^{-1}.
    \]

    \item Fix an orbifold Hermitian line bundle $E_0$ over $\check{S}^2$ with Seifert invariant
    \[
    \left(0,\, \frac{a_1 - 1}{2},\, \dots,\, \frac{a_n - 1}{2} \right).
    \]
    By construction, we have
    \[
    E_0^{-1} \cong E_0 \otimes K_{\check{S}^2}^{-1}.
    \]
    This implies that the spin$^c$ structure
    \[
    S_Y := \pi^* E_0 \otimes W_c
    \]
    on $Y$ is spin, and hence is isomorphic to $\mathfrak{s}_0$.
\end{itemize}

In this setting, Mrowka--Ozsv\'ath--Yu~\cite{MOY96} establish the following natural correspondence:
\[
\bigcup_{\substack{0 \leq \deg (E) < \frac{\deg \left(K_{\check{S}^2}\right)}{2} \\
[E] = [E_0] \in \operatorname{Pic}^t(\check{S}^2)/\mathbb{Z}[L_Y]}} C_+(E) \sqcup C_-(E)
\xrightarrow{\ \Phi\ } M^*_v = M^+_v \sqcup M^-_v \xrightarrow{\ \pi^*\ } M^*_Y(\pi^* E_0 \otimes W_c),
\]
where:
\begin{itemize}
    \item $M^*_Y(\pi^* E_0 \otimes W_c)$ is the space of irreducible critical points of the Chern--Simons--Dirac functional associated to the metric $g_Y$ and the connection for the unique spin structure on $Y$.
    
    \item $M^*_v$ is the moduli space of solutions to the vortex equations:
    \begin{align*}
        \begin{cases}
            2 F_{B_0} - F_{K_{\check{S}^2}} = i\left( |\alpha_0|^2 - |\beta_0|^2 \right) \mu_{\check{S}^2}, \\
            \bar{\partial}_{B_0} \alpha_0 = 0,\quad \bar{\partial}_{B_0}^* \beta_0 = 0, \\
            \alpha_0 = 0 \quad \text{or} \quad \beta_0 = 0,
        \end{cases}
    \end{align*}
    where:
    \begin{itemize}
        \item $\mu_{\check{S}^2}$ is the volume form induced by $g_{\check{S}^2}$,
        \item $\alpha_0$ and $\beta_0$ are orbifold sections of $E_0$ and $E_0^{-1} \otimes K_{\check{S}^2}$, respectively,
        \item $B_0$ is a connection on $E_0$.
    \end{itemize}
    The moduli space is taken modulo orbifold gauge transformations over $\check{S}^2$. The subspaces $M^+_v$ and $M^-_v$ consist of solutions of the form $(\alpha_0, 0)$ and $(0, \beta_0)$, respectively.

    \item $C_+(E)$ denotes the space of effective orbifold divisors associated to the line bundle $E$. The component $C_+(E)$ corresponds to $M^+_v$, while $C_-(E)$ is a copy of $C_+(E)$ corresponding to $M^-_v$.

    \item The symbol $\deg$ denotes the orbifold degree of the line bundles, and $L_Y$ is the line bundle determined by the circle bundle $\pi \colon Y \to \check{S}^2$.
\end{itemize}

Let us denote by $\tau_0 := \tau^{2^{k-1}} \colon Y \to Y$.  
Since both the metric $g_Y$ and the connection $\nabla$ are $\tau_0$-invariant, if we take an anti-complex linear lift 
\[
\widetilde{\tau}_0 \colon S_Y \to S_Y
\]
of $\tau_0$ that is compatible with Clifford multiplication, then we obtain a natural involution
\[
\widetilde{\tau}_0^* \colon M^*_Y(\pi^* E_0 \otimes W_c) \to M^*_Y(\pi^* E_0 \otimes W_c).
\]
What we aim to compute is the fixed point set 
\[
M^*_Y(\pi^* E_0 \otimes W_c)^{\widetilde{\tau}_0^*}.
\]
On the other hand, since the correspondence 
\[
M^*_v = M^+_v \sqcup M^-_v \xrightarrow{\ \pi^*\ } M^*_Y(\pi^* E_0 \otimes W_c)
\]
is natural, we need to consider the same problem for the moduli space of vortices $M^*_v$.
Now observe that the induced map on the base, 
\[
\bar{\tau}_0 := \bar{\tau}^{2^{k-1}} \colon \check{S}^2 \to \check{S}^2,
\]
is a rotation by $\pi$ around the $z$-axis. It follows that:
\begin{align}\label{ooo}
\begin{cases}
    (\bar{\tau}_0)^* E_0 \cong E_0, \\
    (\bar{\tau}_0)^* \left(E_0 \otimes K^{-1}_{\check{S}^2} \right) \cong E_0 \otimes K^{-1}_{\check{S}^2}.
\end{cases}
\end{align}

By composing the conjugation maps $E_0 \to E_0^{-1} \cong E_0 \otimes K^{-1}_{\check{S}^2}$ and 
$E_0 \otimes K^{-1}_{\check{S}^2} \to E_0^{-1} \otimes K_{\check{S}^2} \cong E_0$ with suitable choices of isomorphisms from~\eqref{ooo}, we obtain an order-two anti-complex linear lift of $\tau_0$:
\[
\widetilde{\tau}_0 \colon E_0 \oplus E_0 \otimes K^{-1}_{\check{S}^2} \to E_0 \oplus E_0 \otimes K^{-1}_{\check{S}^2}.
\]
One can verify that this action interchanges $M_v^+$ and $M_v^-$, i.e., $\widetilde{\tau}_0$ maps $M_v^\pm$ to $M_v^\mp$. We now consider the commutative diagram:
\[
\begin{CD}
S_Y = \pi^* E_0 \otimes W_c @>{\pi}>> E_0 \oplus E_0 \otimes K^{-1}_{\check{S}^2} \\
@V{\widetilde{\tau} := \pi^* \circ \widetilde{\tau}_0}VV @V{\widetilde{\tau}_0}VV \\
S_Y = \pi^* E_0 \otimes W_c @>{\pi}>> E_0 \oplus E_0 \otimes K^{-1}_{\check{S}^2}
\end{CD}
\]
which induces the corresponding diagram on moduli spaces:
\[
\begin{CD}
M_v^* @>{\pi^*}>> M_Y^*(\pi^* E_0 \otimes W_c) \\
@V{\widetilde{\tau}_0^*}VV @V{\widetilde{\tau}^*}VV \\
M_v^* @>{\pi^*}>> M_Y^*(\pi^* E_0 \otimes W_c)
\end{CD}
\]
Since we have shown that the fixed point set $(M_v^*)^{\widetilde{\tau}_0^*}$ is empty, it follows that
\[
M_Y^*(\pi^* E_0 \otimes W_c)^{(\pi^* \circ \widetilde{\tau}_0)^*} = \emptyset.
\]
The reducible locus is given by
\[
\{\theta\} = \operatorname{Pic}(\check{S}^2) \to M_Y^{\mathrm{red}}(\pi^* E_0 \oplus W_c),
\]
and is preserved by the action of $\widetilde{\tau}^*$. Here, $\theta$ denotes the unique reducible solution corresponding to the spin structure.

Moreover, it was shown in~\cite[Proposition 5.8.4]{MOY96} that $\theta$ is non-degenerate, meaning that the kernel of the operator
\[
(* d,\ \slashed{\partial}_C) \colon \operatorname{Ker} d^* \times \Gamma(S_Y) \to \operatorname{Ker} d^* \times \Gamma(S_Y)
\]
is zero at $(C, 0) = \theta$. Here, $*$ denotes the Hodge star operator, $S_Y$ is the spinor bundle associated to the spin structure on $Y$, $d^*$ is the formal adjoint of the exterior derivative $d$, and $\slashed{\partial}_C$ denotes the Dirac operator with respect to $g_Y$, the connection $C$, and $\nabla$.
Since all data are $\widetilde{\tau}^*$-equivariant, taking the fixed point set still yields trivial kernel, i.e., the kernel of the operator
\[
(* d^{\widetilde{\tau}^*},\ \slashed{\partial}_C^{\widetilde{\tau}^*}) \colon 
\left( \operatorname{Ker} d^* \times \Gamma(S_Y) \right)^{\widetilde{\tau}^*} \to 
\left( \operatorname{Ker} d^* \times \Gamma(S_Y) \right)^{\widetilde{\tau}^*}
\]
is also zero.

Now, by the construction of the real Floer homotopy type $SWF_R$ given in~\cite{KMT21}, we conclude that
\[
SWF_R(Y, \tau^{2^{k-1}}, \mathfrak{s}_0) = \left(\mathbb{C}^{m}\right)^+
\]
with respect to the fixed data $(g_Y, \nabla)$ for some $m \in \mathbb{Q}$.

On the other hand, the geometric data $(g_Y, \nabla)$ can be connected by a $1$-parameter family of $\tau$-invariant data to $(g_Y, \nabla_{g_Y})$, where $\nabla_{g_Y}$ is the Levi--Civita connection of $g_Y$. By the usual proof of the invariance of Floer homotopy types, the above computation remains valid even for the Levi--Civita connection. This completes the proof.
\end{proof}

\subsection{Computation via Dai-Sasahira-Stoffregen}
Let $k$ be a positive integer. Let us consider the covering space 
\[
Y = \Sigma_{2^k}(T_{2^k\cdot m,\, q})
\]
with the involution $\tau$, where $m$ and $q$ are odd integers and coprime. We claim the following:

\begin{thm} \label{thm: higher sphericity of torus knots}
Let $k, p, q$ be positive integers, where $p$ and $q$ are odd and coprime. Then
\[
SWF_R^{(k)}(T_{2^k p,\, q}) = \left( \mathbb{C}^{\frac{1}{2} \bar\mu(\Sigma(2^k,\, 2^k p,\, q))} \right)^+,
\]
and in particular,
\[
\kappa_R^{(k)}(T_{2^k p,\, q}) = -\frac{1}{2} \bar\mu(\Sigma(2^k,\, 2^k p,\, q)).
\]
\end{thm}

We will compute the invariant 
$\delta^{(k)}_R(K) = \delta_R(\Sigma_{2^k}(K),\, \mathfrak{s}_0,\, \tau)$
using the following proposition, which is essentially proven in~\cite{KPT:2024}.

\begin{prop}
\label{o2homotopy}  
Fix $k \geq 1$. Let $K$ be a knot in $S^3$, and let $\Sigma_{2^k}(K)$ denote its $2^k$-fold branched cover. Suppose there exists an almost-rational plumbing graph $\Gamma$ and a diffeomorphism $\partial W_\Gamma \cong \Sigma_{2^k}(K)$, where $W_\Gamma$ denotes the plumbed 4-manifold defined by $\Gamma$, such that the $2^{k-1}$ power of the deck transformation on $\Sigma_{2^k}(K)$ extends smoothly to an involution $\tau$ on $W_\Gamma$.
Assume further that there exists an almost $I$-invariant path $\gamma$ that carries the lattice homology of $(\Gamma, \mathfrak{s}_0)$, where $\mathfrak{s}_0$ denotes the unique spin structure on $\Sigma_{2^k}(K)$. Then there exists an $O(2)$-equivariant map
\[
\mathcal{T}^{O(2)} \colon \mathcal{H}(\gamma, \mathfrak{s}_0) \to SWF(\Sigma_{2^k}(K), \mathfrak{s}_0)
\]
which is an $S^1$-equivariant homotopy equivalence, with respect to a certain $O(2)$-action on $\mathcal{H}(\gamma, \mathfrak{s}_0)$.\footnote{Here, $\mathcal{H}(\gamma, \mathfrak{s}_0)$ denotes the \emph{$S^1$-path homotopy type} of $\gamma$; see~\cite[Definition 3.2]{DSS2023} for its definition.} 
\end{prop}

We briefly review the definition of an almost $I$-invariant path.

\begin{defn}
Let $\Gamma$ be an almost-rational negative-definite plumbing graph, and let $W_\Gamma$ denote the associated plumbed 4-manifold. Suppose $\tau$ is an orientation-preserving involution on $W_\Gamma$ with codimension-two fixed point set, and let $\mathfrak{s}$ be a spin structure on $\partial W_\Gamma$ such that $\tau^* \mathfrak{s} = \bar{\mathfrak{s}} = \mathfrak{s}$.

An \emph{almost $I$-invariant path} is a sequence of spin$^c$ structures on $W_\Gamma$,
\[
\gamma = \{ \mathfrak{s}_{-n}, \dots, \mathfrak{s}_{-1}, \mathfrak{s}_1, \dots, \mathfrak{s}_n \},
\]
satisfying the following conditions:
\begin{itemize}
    \item For all $i = 1, \dots, n$, the restriction $\mathfrak{s}_i|_{\partial W_\Gamma} = \mathfrak{s}$ and $\mathfrak{s}_{-i}|_{\partial W_\Gamma} = \mathfrak{s}$;
    
    \item For all $i = 1, \dots, n$, we have $\mathfrak{s}_{-i} = \tau^* \bar{\mathfrak{s}}_i$;
    
    \item For each $i = 1, \dots, n$, we have
    \[
    \mathfrak{s}_{i+1} - \mathfrak{s}_i = \operatorname{PD}[S]
    \]
    for a sphere $S \subset W_\Gamma$ representing a node of $\Gamma$;
    
    \item For each $i = 1, \dots, n$, we have
    \[
    \mathfrak{s}_{-i} - \mathfrak{s}_{-i-1} = \operatorname{PD}[S]
    \]
    for a sphere $S \subset W_\Gamma$ representing a node of $\Gamma$;
    
    \item The difference between the two central terms is given by
    \[
    \mathfrak{s}_1 - \mathfrak{s}_{-1} = \sum_{S \in \mathcal{S}} \operatorname{PD}[S],
    \]
    for a finite collection $\mathcal{S}$ of pairwise disjoint smoothly embedded spheres $S \subset W_\Gamma$ with $[S]^2 < 0$, such that each $S$ is setwise fixed by $\tau$, and $\tau|_S$ is either the identity or an orientation-preserving involution fixing two points.
\end{itemize}
\end{defn}
Recall that given a $\mathrm{spin}^c$ structure $\mathfrak{s}$ on $\partial W_\Gamma$, a path of $\mathrm{spin}^c$ structures on $W_\Gamma$ is said to \emph{carry the lattice homology of $(\Gamma, \mathfrak{s})$} if the natural inclusion map
\[
\mathcal{H}(\gamma, \mathfrak{s}) \hookrightarrow \mathcal{H}(\Gamma, \mathfrak{s})
\]
induces a chain homotopy equivalence on the $S^1$-equivariant Borel chain complexes.

\begin{proof}[Proof of \Cref{o2homotopy}]
The existence of the desired $O(2)$-equivariant map follows by essentially the same argument as in the proof of~\cite[Theorem 1.3]{KPT:2024}, since the description of $\partial W_\Gamma$ as a 2-fold branched cover is not used there. Therefore, we omit the details.
\end{proof}

\begin{cor} \label{cor: froyshov is NS}
    Under the same assumption of \cref{o2homotopy}, 
    \[
    \delta^{(k)}_R(K)= \frac{1}{16} \left(c_1(\mathfrak{s}_1)^2-\sigma (W_\Gamma)\right) .  \eqno\QEDB
    \]
\end{cor}
\begin{proof}
    This can be shown easily by following the proof of \cite[Theorem 4.7]{KPT:2024}.
\end{proof}

We can now prove \Cref{thm: higher sphericity of torus knots}.
\begin{proof}[Proof of \Cref{thm: higher sphericity of torus knots}]
Consider the plumbing graph $\Gamma_{k,m,q}$ introduced in \Cref{subsec: brieskorn k_m_q}. Recall that we have constructed a smooth $\mathbb{Z}_{2^k} = \langle \tau \rangle$-action on the associated plumbed 4-manifold $W_{\Gamma_{k,m,q}}$, such that the boundary $\partial W_{\Gamma_{k,m,q}}$ is $\mathbb{Z}_{2^k}$-equivariantly diffeomorphic to $\Sigma_{2^k}(T_{2^k \cdot m, q})$, where $\mathbb{Z}_{2^k}$ acts by deck transformations.
Denote by $\mathfrak{s}$ the unique spin structure on $\partial W_{\Gamma_{k,m,q}}$, or equivalently, the unique self-conjugate $\mathrm{spin}^c$ structure.

   We claim that there exists an almost $I$-invariant path of $\mathrm{spin}^c$ structures on $W_{\Gamma_{k,m,q}}$ that carries the lattice homology of $(\Gamma_{k,m,q}, \mathfrak{s})$. To prove this claim, we follow the argument in~\cite[Lemma 4.5]{KPT:2024}; the only point that needs to be checked is that the spherical Wu class $\mathrm{Wu}(\Gamma_{k,m,q}, \mathfrak{s})$ for the plumbing graph $\Gamma_{k,m,q}$, with respect to the $\mathrm{spin}^c$ structure $\mathfrak{s}$, can be written as
\[
\mathrm{Wu}(\Gamma_{k,m,q}, \mathfrak{s}) = \sum_{S \in \mathcal{S}} \operatorname{PD}[S],
\]
for some finite collection $\mathcal{S}$ of pairwise disjoint, setwise $\mathbb{Z}_{2^k}$-fixed, smoothly embedded spheres in $W_{\Gamma_{k,m,q}}$, such that $\mathbb{Z}_{2^k}$ acts on each sphere in $\mathcal{S}$ either trivially or by rotation.
To see this, recall that the Wu class can be expressed as
\[
\mathrm{Wu}(\Gamma_{k,m,q}, \mathfrak{s}) = \sum_{v \in V} \operatorname{PD}[S_v],
\]
where $V$ is a set of zero-sections of disk bundles associated to certain nodes of $\Gamma_{k,m,q}$, such that no two elements of $V$ are adjacent. We define a $\mathbb{Z}_{2^k}$-action on the set of nodes of $\Gamma_{k,m,q}$ by
\[
\tau(a_i) = a_i, \qquad \tau(c) = c, \qquad \tau(b_{j,\gamma}) = b_{j,\gamma+1},
\]
so that $\tau(S_v) = S_{\tau(v)}$ for every node $v$ of $\Gamma_{k,m,q}$. Since the spin structure $\mathfrak{s}$ is $\mathbb{Z}_{2^k}$-invariant, the set $V$ must also be setwise $\tau$-invariant; that is, a node $v$ belongs to $V$ if and only if $\tau(v)\in V$.
We may thus write
\[
V = V_{a,c} \sqcup V_b \sqcup \tau V_b \sqcup \cdots \sqcup \tau^{2^k - 1} V_b,
\]
where
\[
V_{a,c} = \{ v \in V \mid S_v = S_{a_i} \text{ for some } i \text{ or } S_v = S_c \}, \qquad
V_b = \{ v \in V \mid S_v = S_{b_j,0} \text{ for some } j \}.
\] We now divide into two cases.

   \textbf{Case 1: $c \in V$ and $V_b \neq \emptyset$.} Let $N_b = \{j_1, \dots, j_t\}$, with $0 < j_1 < \cdots < j_t \leq s$, be the set of integers such that
\[
V_b = \{ \text{nodes } v \text{ of } \Gamma_{k,m,q} \mid S_v = S_{b_j} \text{ for some } j \in N_b \}.
\]
For each $u = 1, \dots, t-1$, choose a smooth path $\gamma_u$ from the south pole of $S_{b_{j_u}}$ to the north pole of $S_{b_{j_{u+1}}}$. Also, choose a smooth path $\gamma$ from the north pole of $S_{b_{j_1}}$ to a point on $S_c$ that is neither the north pole nor the south pole. 
We perturb these paths so that the following set of paths are pairwise disjoint and do not intersect any of the spheres $S_{a_1}, \dots, S_{a_r}$:
\[
\gamma,\, \tau\gamma,\, \dots,\, \tau^{2^k - 1} \gamma,\quad \gamma_1,\, \tau\gamma_1,\, \dots,\, \tau^{2^k - 1} \gamma_1,\quad \dots,\quad \gamma_{t-1},\, \tau\gamma_{t-1},\, \dots,\, \tau^{2^k - 1} \gamma_{t-1}.
\]
We now tube the following pairwise disjoint smoothly embedded spheres along the paths above:
\[
S_c,\, S_{b_{j_1}},\, \tau S_{b_{j_1}},\, \dots,\, \tau^{2^k - 1} S_{b_{j_1}},\, \dots,\, S_{b_{j_t}},\, \tau S_{b_{j_t}},\, \dots,\, \tau^{2^k - 1} S_{b_{j_t}}.
\]
This yields a smoothly embedded sphere $S$ in $W_{\Gamma_{k,m,q}}$ that is disjoint from the spheres $S_{a_1}, \dots, S_{a_r}$, is setwise fixed by the $\mathbb{Z}_{2^k}$-action, and on which $\mathbb{Z}_{2^k}$ acts by rotation. Since we have
\[
\mathrm{Wu}(\Gamma_{k,m,q}, \mathfrak{s}) = [S] + \sum_{v \in V_{a,c} \setminus \{c\}} \operatorname{PD}[S_v],
\]
the claim is proven in this case.

\textbf{Case 2: $c \notin V$ and $V_b \neq \emptyset$.} The argument proceeds similarly to Case 1, except that instead of taking the path $\gamma$ from the north pole of $S_{b_{j_1}}$ to a non-polar point on $S_c$, we introduce a small auxiliary sphere $S^0_c \subset W_{\Gamma_{k,m,q}}$ centered at the south pole of $S_c$. We then choose $\gamma$ to be a smooth path from the north pole of $S_{b_{j_1}}$ to a point on $S^0_c$.
This construction again yields a smoothly embedded sphere $S \subset W_{\Gamma_{k,m,q}}$ that is setwise fixed under the $\mathbb{Z}_{2^k}$-action, and on which $\mathbb{Z}_{2^k}$ acts by rotation. Moreover, we have
\[
\mathrm{Wu}(\Gamma_{k,m,q}, \mathfrak{s}) = [S] + \sum_{v \in V_a} \operatorname{PD}[S_v],
\]
so the claim also holds in this case.

    \textbf{Case 3: $V_b = \emptyset$.} In this case, for every $v \in V$, the sphere $S_v$ is already setwise $\mathbb{Z}_{2^k}$-fixed, and the induced action of $\mathbb{Z}_{2^k}$ on $S_v$ is either trivial or by rotation. Hence the claim also holds in this case.

Now, using the claim just proven, we can choose an almost $I$-invariant path $\gamma = \{ \dots, \mathfrak{s}_{-1}, \mathfrak{s}_1, \dots \}$ of $\mathrm{spin}^c$ structures on $W_{\Gamma_{k,m,q}}$ which carries the lattice homology of $(\Gamma_{k,m,q}, \mathfrak{s})$. Therefore, the assumptions of \Cref{o2homotopy} are satisfied, and \Cref{cor: froyshov is NS} gives
\[
\delta^{(k)}(T_{2^k \cdot m, q}) = \frac{c_1(\mathfrak{s}_1)^2 - \sigma(W_{\Gamma_{k,m,q}})}{16}.
\]
From the discussion in \cite[Section 6.1]{DSS2023}, and from the construction of an almost $I$-invariant path from an almost $J$-invariant path (as in the proof of \cite[Lemma 4.5]{KPT:2024}), we know that $\mathfrak{s}_1$ lies in the \emph{$J$-invariant lattice cube}; that is,
\[
c_1(\mathfrak{s}_1) = \sum_{v \in V} \lambda_v [S_v], \qquad \lambda_v \in \{1, -1\}.
\]
Since no two nodes in $V$ are adjacent, we compute
\[
c_1(\mathfrak{s}_1)^2 = \sum_{v \in V} \lambda_v^2 [S_v]^2 = \sum_{v \in V} [S_v]^2 = \mathrm{Wu}(\Gamma_{k,m,q}, \mathfrak{s})^2.
\]
Thus, we deduce that
\[
\delta^{(k)}(T_{2^k \cdot m, q}) = \frac{\mathrm{Wu}(\Gamma_{k,m,q}, \mathfrak{s})^2 - \sigma(W_{\Gamma_{k,m,q}})}{16} = -\frac{1}{2} \bar\mu(W_{\Gamma_{k,m,q}}, \mathfrak{s}) = -\frac{1}{2} \bar\mu(\Sigma(2^k, 2^k \cdot m, q)).
\]
On the other hand, it follows directly from \Cref{sphericality} that
\[
SWF_R(\Sigma_{2^k}(T_{2^k \cdot m, q}), \tau^{2^{k-1}}) = \left( \mathbb{C}^{-\alpha} \right)^+
\]
for some $\alpha \in \mathbb{Q}$, which implies that
\[
\kappa^{(k)}_R(T_{2^k \cdot m, q}) = \delta^{(k)}(T_{2^k \cdot m, q}) = \alpha.
\]
Therefore, we conclude that
\[
\kappa^{(k)}_R(T_{2^k \cdot m, q}) = -\frac{1}{2} \bar\mu(\Sigma(2^k, 2^k \cdot m, q)),
\]
as desired.
\end{proof}

\section{Proof of the main theorem}
To prove the main theorem, we begin with the following elementary observation.

\begin{lem} \label{lem: linear combination of spherical knots}
Let $k$ be a positive integer, and let $K_1, \dots, K_n$ be $k$-strongly spherical knots. Then, for any integers $c_1, \dots, c_n$, the knot
\[
K = c_1 K_1 \# \cdots \# c_n K_n
\]
is $k$-strongly spherical and satisfies
\[
\kappa_R^{(k)}(K) = \sum_{i=1}^n c_i \kappa_R^{(k)}(K_i).
\]
\end{lem}

\begin{proof}
Since each $K_i$ is $k$-strongly spherical, we have
\[
SWF_R^{(k)}(K_i) \simeq \left( \mathbb{C}^{-\kappa_R^{(k)}(K_i)} \right)^+ \quad \text{for } i = 1, \dots, n.
\]
Using the multiplicativity of the suspension spectrum under connected sums and wedge powers, we compute:
\[
SWF_R^{(k)}(K) \simeq \bigwedge_{i=1}^n \left(SWF_R^{(k)}(K_i)\right)^{\wedge c_i} \simeq \left( \mathbb{C}^{-\sum_{i=1}^n c_i \kappa_R^{(k)}(K_i)} \right)^+.
\]
This implies that $K$ is $k$-strongly spherical and satisfies the claimed formula for $\kappa_R^{(k)}(K)$.
\end{proof}

We are finally ready to prove the main theorem.

\begin{thm} \label{thm: kappa inequality}
Let $K$ be a knot that can be transformed into a slice knot by applying full negative twists along two disjoint disks,  
where one intersects $K$ algebraically once and the other intersects it algebraically three times.  
An example of such a knot is the figure-eight knot $4_1$.  
Then, for any positive integers $k$, $m$, and $c$, with $m$ odd, we have
\[
\frac{1}{2} \le \kappa^{(k)}_R\left(-cK_{2^k \cdot m,1}\right).
\]
\end{thm}

\begin{proof}
Write $p = 2^k \cdot m$. By assumption, there exists a smooth concordance $S$, of homology class $$(p, 3p, \dots, p, 3p)$$ in $2c \mathbb{CP}^2$, from $c K_{p,1}$ to $-c T_{p,10p-1}$. Equivalently, there exists a smooth concordance, again denoted by $S$, from $cT_{p,10p-1}$ to $-c K_{p,1}$ with the same homology class in $2c \mathbb{CP}^2$.

By the same argument as in the proof of \Cref{lem: homological quantities}, we have
\[
b^+(\Sigma_{2^k}(S)) = c(2^k + 1), \qquad b^+(\Sigma_{2^{k-1}}(S)) = c(2^{k-1} + 1), \qquad \sigma(\Sigma_{2^k}(S)) = 2c,
\]
and therefore
\[
b^+(\Sigma_{2^k}(S)) - b^+(\Sigma_{2^{k-1}}(S)) = c \cdot 2^{k-1} > 0.
\]
Furthermore, $cT_{p,10p-1}$ is $k$-strongly spherical by \Cref{thm: higher sphericity of torus knots}. Hence we obtain the following inequality from \Cref{thm: higher real 10/8}:
\begin{equation} \label{eq:finalinequal}
    -\frac{\sigma(\Sigma_{2^k}(S))}{16} + \frac{1}{2} \le b^+(\Sigma_{2^k}(S)) - b^+(\Sigma_{2^{k-1}}(S)) + \kappa_R^{(k)}(-c K_{p,1}) - \kappa_R^{(k)}(c T_{p,10p-1}).
\end{equation}

In addition, using \Cref{lem:mu-bar-special-case}, \Cref{thm: higher sphericity of torus knots}, and \Cref{lem: linear combination of spherical knots}, we compute
\[
\kappa_R^{(k)}(c T_{p,10p-1}) = c\kappa_R^{(k)}(T_{2^k m,\, 10 \cdot 2^k m - 1}) = -\frac{c}{2} \bar{\mu}(\Sigma(2^k,\, 2^k m,\, 2^k \cdot 10m - 1)) = -\frac{c}{2} \left( -2^k - \frac{1}{4} \right).
\]
Substituting these into \eqref{eq:finalinequal}, we obtain
\[
-\frac{c}{8} + \frac{1}{2} \le c \cdot 2^{k-1} + \kappa_R^{(k)}(-c K_{p,1}) + \frac{c}{2} \left( -2^k - \frac{1}{4} \right).
\]
The theorem follows.
\end{proof}

\begin{proof}[Proof of \Cref{thm:main}]
    Suppose that $(4_1)_{p,q}$ has finite order in the smooth knot concordance group $\mathcal{C}$, where $p$ and $q$ are nonzero coprime integers and $p \neq \pm 1$ (i.e., the $(p,q)$-cable is nontrivial). Since $-(4_1)_{p,q} = (4_1)_{p,-q}$, we may assume without loss of generality that $p, q > 0$.  
The formula for signatures of cable knots given by Litherland~\cite[Theorem 2]{Litherland:1979} implies that the Levine–Tristram signature function of $(4_1)_{p,q}$ coincides with that of the torus knot $T_{p,q}$, so we have that $q = 1$~\cite[Lemma 1]{Litherland:1979}.  
It is known from \cite[Theorem 1.3]{HKPS:2022-1} that $(4_1)_{p,1}$ has infinite order in $\mathcal{C}$ whenever $p$ is odd, so we deduce that $p$ must be even.  
However, it follows from \Cref{thm: kappa inequality} that $(4_1)_{p,1}$ has infinite order in $\mathcal{C}$ for every nonzero even integer $p$, yielding a contradiction.  
The theorem follows.\end{proof}

\bibliographystyle{alpha}
\bibliography{knotbib}

\end{document}